\documentclass{article}

\usepackage[utf8]{inputenc}

\usepackage{amssymb}
\usepackage{amsmath}
\usepackage{amsthm}
\usepackage{latexsym}
\usepackage{textcomp}
\usepackage{epsfig}
\usepackage{subfigure}
\usepackage{xspace}
\usepackage{xcolor}
\usepackage{thm-restate}

\newcommand\numberthis{\addtocounter{equation}{1}\tag{\theequation}}

\usepackage{todonotes}

\usepackage{paralist} 
\usepackage{wrapfig}
\usepackage{tabularkv}
\usepackage{authblk}
\usepackage{lineno}
\usepackage[pdfencoding=auto, psdextra]{hyperref}
\usepackage{cleveref}

\newtheorem{property}{Property}
\newtheorem{lemma}{Lemma}
\newtheorem{theorem}{Theorem}
\newtheorem{corollary}{Corollary}

\Crefname{property}{Property}{Properties}
\Crefname{lemma}{Lemma}{Lemmas}
\Crefname{corollary}{Corollary}{Corollaries}
\Crefname{figure}{Fig.}{Figs.}
\Crefname{equation}{Eq.}{Eqs.}

\graphicspath{{figures/}}

\newcommand{\remove}[1]{}

\begin{document}

\markboth{C. Binucci, E. Di Giacomo, M. Kaufmann, G. Liotta, and A. Tappini}
{$k$-planar Placement and Packing of $\Delta$-regular Caterpillars}

\title{$k$-planar Placement and Packing\\ of $\Delta$-regular Caterpillars
}

\author[1]{Carla Binucci}
\author[1]{Emilio Di Giacomo}
\author[2]{Michael Kaufmann}
\author[1]{Giuseppe Liotta}
\author[1]{Alessandra Tappini}
\affil[1]{Dipartimento di Ingegneria, Universit\`a degli Studi di Perugia, via G. Duranti 93, 06125, Perugia, Italy. \texttt{\{carla.binucci, emilio.digiacomo, giuseppe.liotta, alessandra.tappini\}@unipg.it}}

\affil[2]{Wilhelm-Schickard Institut f{\"u}r Informatik, Universit{\"a}t T{\"u}bingen, Sand 13, 72076, T{\"u}bingen, Germany. \texttt{mk@informatik.uni-tuebingen.de}}

\maketitle

\begin{abstract}
 This paper studies a \emph{packing} problem in the so-called beyond-planar setting, that is when the host graph is ``almost-planar'' in some sense. Precisely, we consider the case that the host graph is $k$-planar, i.e., it admits an embedding with at most $k$ crossings per edge, and focus on families of $\Delta$-regular caterpillars, that are caterpillars whose non-leaf vertices have the same degree $\Delta$. We study the dependency of $k$ from the number $h$ of caterpillars that are packed, both in the case that these caterpillars are all isomorphic to one another (in which case the packing is called \emph{placement}) and when they are not. We give necessary and sufficient conditions for the placement of $h$ $\Delta$-regular caterpillars and sufficient conditions for the packing of a set of $\Delta_1$-, $\Delta_2$-, $\dots$, $\Delta_h$-regular caterpillars
 such that the degree $\Delta_i$ and the degree $\Delta_j$ of the non-leaf vertices can differ from one caterpillar to another, for  $1 \leq i,j \leq h$, $i\neq j$. 
\end{abstract}

\section{Introduction}

Graph \emph{packing} is a classical problem in graph theory. The original formulation requires to merge several smaller graphs into a larger graph, called the \emph{host~graph}, without creating multiple edges. More precisely, graphs $G_1,G_2,\dots,G_h$ with $G_i = (V_i,E_i)$
should be combined to a new graph $G=(V,E)$ by injective mappings $\eta_i: V_i \rightarrow V$  so that
$V=V_1 \cup V_2 \cup \dots \cup V_h$ and the images of the edge sets $E_i$ do not intersect.
It has been often assumed that $|V_i| = n$ for all $i = 1,2, \dots h$, and thus the mappings $\eta_i$ are bijective. 
Many combinatorial problems can be regarded as packing problems. For example, the Hamiltonian cycle problem for~a graph $G$ can be stated as
the problem of packing an $n$-vertex cycle with the complement~of~$G$. 

When no restriction is imposed on the host graph, we say that the host graph is~$K_n$. Some classical results in this setting are those by Bollob\'{a}s and Eldridge~\cite{DBLP:journals/jct/BollobasE78}, Teo and Yap \cite{DBLP:journals/gc/TeoY90}, Sauer and Spencer \cite{DBLP:journals/jct/SauerS78}, while related famous conjectures are by Erd\H{o}s and S\'os from 1963~\cite{Erdos64} and by Gy\'arf\'as from 1978~\cite{gyarfas1978packing}. Within this line of research, Wang and Sauer~\cite{WANG1993137}, and Mah\'{e}o et al.~\cite{DBLP:journals/ejc/MaheoSW96} characterized triples of trees that admit a packing into $K_n$. Haler and Wang~\cite{DBLP:journals/ajc/HalerW14} extended this result to four copies of a tree.
Further notable work on graph packing into $K_n$ is by Hedetniemi et al.~\cite{hhs-npttk-81}, Wozniak and Wojda~\cite{DBLP:journals/gc/WozniakW93} and Aichholzer~et~al.~\cite{AICHHOLZER201735}.
A packing problem with identical copies of a graph is also called a \emph{placement problem} (see, e.g.,~\cite{DBLP:journals/ajc/HalerW14,WANG1993137,DBLP:journals/dm/Zak11}).

A tighter relation to graph drawing was established when researchers did not consider $K_n$ to be the host graph, but required that the host graph is planar.  
The main question here is how to pack two trees of size $n$ into a planar graph of size~$n$. After a long series of intermediate steps
\cite{DBLP:conf/cccg/Frati09,DBLP:journals/ipl/FratiGK09,DBLP:journals/jgt/GarciaHHNT02,DBLP:conf/wads/GeyerHKKT13,oo-tpptt-06} where the class of trees that could be packed has been gradually generalized,  
Geyer et al. \cite{DBLP:journals/jocg/GeyerHKKT17} showed that any two non-star trees can be embedded into a planar graph.

Relaxing the planarity condition allows for packing of more (than two) trees, and restricting the number of crossings for each edge, i.e., in the so-called beyond-planar setting~\cite{DBLP:journals/csur/DidimoLM19,DBLP:books/sp/20/HT2020,KOBOUROV201749}, still keeps the host graph sparse. 
The study of the packing problem in the beyond planarity setting was started by De Luca et al.~\cite{DBLP:journals/jgaa/LucaGHKLLMTW21}, who consider how to pack caterpillars, paths, and cycles into 1-planar graphs (see, e.g.,~\cite{KOBOUROV201749} for a survey and references on 1-planarity). While two trees can always be packed into a planar graph, it may not be possible to pack three trees into a $1$-planar graph.

\smallskip 

In this work we further generalize the problem by allowing the host graph to be $k$-planar for any $k \geq 1$, and we study the dependency of $k$ on the number of caterpillars to be packed and on their vertex degree. We consider \emph{$\Delta$-regular} caterpillars, which are caterpillars whose non-leaf vertices all have the same degree. Our results can be briefly outlined as follows.
\begin{itemize}
    \item We consider the packing problem of $h$ copies of the same $\Delta$-regular caterpillar into a $k$-planar graph. 
We characterize those families of $h$ $\Delta$-regular caterpillars which admit a placement into a $k$-planar graph and show that $k \in O(\Delta h+h^2)$.
    \item We extend the study from the placement problem to the packing problem by considering 
    a set of $\Delta_1$-, $\Delta_2$-, $\dots$, $\Delta_h$-regular caterpillars
    such that the degree $\Delta_i$ and the degree $\Delta_j$ of the non-leaf vertices can differ from one caterpillar to another, with $1 \leq i,j \leq h$, $i\neq j$. By extending the techniques of the bullet above, we give sufficient conditions for the existence of a $k$-planar packing of these caterpillars and show that $k \in O(\Delta h^2)$.
    \item Finally, we prove a general lower bound on $k$ and show that this lower bound can be increased for small values of $h$ and for caterpillars that are not $\Delta$-regular.   
\end{itemize}

The rest of the paper is organized as follows. Preliminaries are in \Cref{se:preliminaries}. The placement of $h$ $\Delta$-regular caterpillars into a $k$-planar graph is discussed in \Cref{se:kplanar-caterpillar}. \Cref{se:packing} is devoted to $k$-planar $h$-packing, while \Cref{se:lower-bounds} gives lower bounds on the value of $k$ as a function of $h$. Concluding remarks and open problems can be found in \Cref{se:open-problems}.

\section{Preliminaries}\label{se:preliminaries}

We assume familiarity with basic graph drawing and graph theory terminology (see, e.g., ~\cite{DBLP:books/ph/BattistaETT99,DBLP:conf/dagstuhl/1999dg,DBLP:books/ws/NishizekiR04}) and recall here only those concepts and notation that will be used in the paper.

Given a graph $G$, we denote by $\deg_{G}(v)$ the degree of a vertex $v$ in $G$. Let $G_1, G_2, \dots, G_h$ be $h$ graphs, all having $n$ vertices, an \emph{$h$-packing} of $G_1, G_2, \dots, G_h$ is an $n$-vertex graph $G$ that contains $G_1, G_2, \dots, G_h$ as edge-disjoint spanning subgraphs. We also say that $G_1, G_2, \dots, G_h$ can be \emph{packed into $G$} and that $G$ is the \emph{host graph} of $G_1, G_2, \dots, G_h$. An $h$-packing of $h$ graphs into a host graph $G$ such that the $h$ graphs are all isomorphic to a graph $H$, is called an \emph{$h$-placement} of $H$ into~$G$. We also say that $G_1, G_2, \dots, G_h$ can be \emph {placed into $G$}. The following property establishes a necessary condition for the existence of an $h$-packing into any host graph.

\begin{property}\label{prop:necessary}
A packing of $h$ connected $n$-vertex graphs exists only if $n \geq 2h$ and $\deg_{G_i}(v)\leq n-h$, for each $i \in \{1,2,\dots,h\}$ and for each vertex $v$.
\end{property}
\begin{proof}
Each $G_i$ has at least $n-1$ edges (because it is connected); thus, if $n < 2h$ the $h$ graphs have more edges in total than the number of edges of any graph with $n$ vertices. But since graphs $G_i$ must be edge-disjoint subgraphs of $G$, the number of edges of $G$ must be at least the total number of edges of the graphs $G_i$.
Since $\deg_{G_i}(v) \geq 1$ for every  $i \in \{1,2,\dots,h\}$ and for each $v$ (because each $G_i$ is connected) and since $\sum_{i=1}^h\deg_{G_i}(v) \leq n-1$ (because $G$ cannot have vertex-degree larger that $n-1$), it holds that $\deg_{G_i}(v)\leq n-h$, for each $i \in \{1,2,\dots,h\}$ and for each vertex $v$. 
\end{proof}

A \emph{$k$-planar graph} is a graph that admits a drawing in the plane such that each edge is crossed at most $k$ times. If the host graph of an $h$-packing ($h$-placement) is $k$-planar, we will talk about  a \emph{k-planar $h$-packing} (\emph{k-planar $h$-placement}). Sometimes, we shall simply say $k$-planar packing or $k$-planar placement, when the value of $h$ is clear from the context or not relevant. 

A \emph{caterpillar} is a tree such that removing all leaves we are left with a path, called \emph{spine}. A caterpillar $T$ is \emph{$\Delta$-regular}, for $\Delta \geq 2$, if $\deg_T(v)=\Delta$ for every vertex $v$ of the spine of $T$. The number of vertices of a  $\Delta$-regular caterpillar is $n=\sigma(\Delta -1)+2$ for some positive integer $\sigma$, which is the number of vertices of the spine.

\section{\texorpdfstring{$h$}--placement of \texorpdfstring{$\Delta$}--regular Caterpillars into \texorpdfstring{$k$}--planar Graphs}\label{se:kplanar-caterpillar}

Given $h$ copies of a same $\Delta$-regular caterpillar, we want to study under which conditions they admit a placement into a $k$-planar graph. We start by showing that the necessary condition stated in \Cref{prop:necessary} is, in general, not sufficient to guarantee a placement even for $\Delta$-regular caterpillars.

\begin{theorem}\label{th:nopacking-2h}
   For every $h \geq 2$, let $\Delta$ be a positive integer such that $\frac{h-1}{\Delta-1}$ is not an integer. A set of $h$ $\Delta$-regular caterpillars with $n=2h$ vertices does not admit a placement into any graph.
\end{theorem}
\begin{proof}
Since each caterpillar has $n-1$ edges and the number of caterpillars is $h=\frac{n}{2}$, the total number of edges is $\frac{n(n-1)}{2}$ and thus, if a placement exists, the host graph can only be $K_n$. We now prove that this is not possible. Denote by $C_1, C_2,\ \dots, C_h$ the $h$ caterpillars and suppose that a packing into $K_n$ exists. Let $v$ be a vertex of $K_n$ and let $v_1,v_2,\dots,v_h$ be the $h$ vertices that are mapped to $v$, with $v_i$ being a vertex of $C_i$. Each vertex $v_i$ has degree in $C_i$ that is either $\Delta$ or $1$ (because each $C_i$ is $\Delta$-regular). Denote by $c$ the number of vertices among $v_1,v_2,\dots,v_h$ that have degree $\Delta$; the degree of $v$ in the packing is $c \Delta +(h-c)$ and since the degree of $v$ in $K_n$ is $n-1$, it must be $c \Delta +(h-c)=n-1$, i.e., $c \Delta +(h-c)=2h-1$, which can be rewritten as $c=\frac{h-1}{\Delta-1}$. But this is not possible because $c$ is integer, while $\frac{h-1}{\Delta-1}$ is not.  
\end{proof}

In the rest of this section we shall establish  necessary and sufficient conditions that characterize when a set of $h$ isomorphic $\Delta$-regular caterpillars admit a $k$-planar $h$-placement.  Concerning the sufficiency, in \Cref{sse:zig-zag} we describe a constructive argument that computes a set of so-called \emph{zig-zag drawings} and study the properties of such drawings. In \Cref{sse:characterization}, we complete the characterization by also giving necessary conditions for an $h$-placement of $\Delta$-regular caterpillars into a $k$-planar graph; in the same section, we give an upper bound on $k$ as a function of $h$ and $\Delta$. 

\smallskip

We recall that a $\Delta$-regular caterpillar has a number of vertices $n$ that is equal to $\sigma(\Delta-1)+2$ for some natural number $\sigma$, which is the number of vertices of the spine. 
While $\Delta$-regular caterpillars are defined for any value of $\sigma \geq 1$, when we want to pack a set of $h \geq 2$ caterpillars, \Cref{prop:necessary} requires that each caterpillar has at least two spine vertices, i.e., that $\sigma \ge 2$ for each caterpillar. Otherwise, the unique spine vertex would have degree $n-1$ and  \Cref{prop:necessary} would not hold. 

\subsection{Zig-zag Drawings of \texorpdfstring{$\Delta$}--regular caterpillars}\label{sse:zig-zag}

\begin{figure}[tbp]
		\centering
     	\subfigure[]{\label{fi:zig-zag-drawing-a}\includegraphics[width=0.32\columnwidth, page=5]{zig-zag-drawing}}
		\hfil
		\subfigure[]{\label{fi:zig-zag-drawing-b}\includegraphics[width=0.32\columnwidth, page=1]{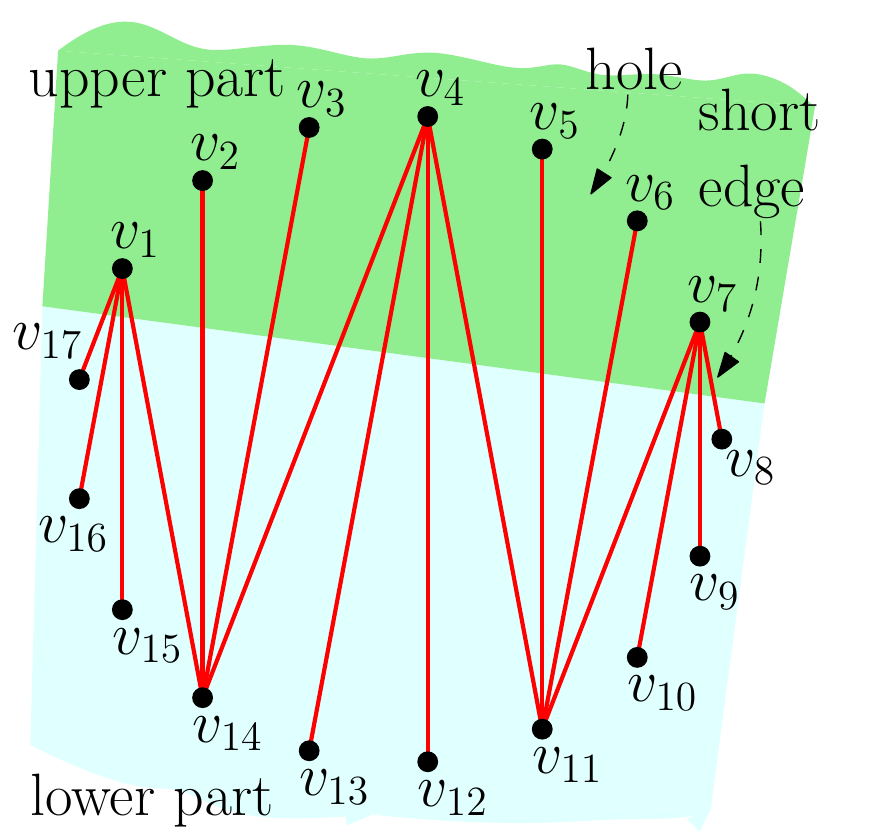}}
		\hfil
		\subfigure[]{\label{fi:zig-zag-drawing-d}\includegraphics[width=0.32\columnwidth, page=2]{zig-zag-drawing}}
		\caption{\label{fi:zig-zag-drawing} (a) A zig-zag drawing of a $4$-regular caterpillar; (b) the upper and the lower part are highlighted; c) a $2$-packing obtained by the drawing of (b) with a copy of it rotated by one step.
		}
	\end{figure}

Let $C$ be a $\Delta$-regular caterpillar with $n$ vertices; we construct a drawing $\Gamma$
of $C$ as shown in \Cref{fi:zig-zag-drawing}. The number of vertices of the spine of $C$ is $\sigma=\frac{n-2}{\Delta-1}$; consider a set of $\sigma$ points on a circle $\gamma$ and denote by $u_1, u_2,\dots,u_\sigma$ these points according to the circular clockwise order they appear along $\gamma$.
Draw the spine of $C$ by connecting, for $i=1,2,\dots,\lfloor \frac{\sigma}{2} \rfloor$, the points $u_i$ and $u_{i+1}$ to the point $u_{\sigma-i+1}$; see \Cref{fi:zig-zag-drawing-a}. If $\sigma$ is even and $i=\frac{\sigma}{2}$, the points $u_{i+1}$ and $u_{\sigma-i+1}$ coincide and therefore the point $u_{\frac{\sigma}{2}}$ is connected only to $u_{\frac{\sigma}{2}+1}$.
Notice that all points $u_i$ have two incident edges, except $u_1$ and $u_{\lfloor \frac{\sigma}{2} \rfloor +1}$ which have only one. We add the leaves adjacent to each vertex $u_i \not \in \{u_1,u_{\lfloor \frac{\sigma}{2} \rfloor +1}\}$ by connecting $u_{\sigma-i+1}$ to $\Delta-2$ points between $u_{i}$ and $u_{i+1}$; we then add the leaves adjacent to $u_1$ by connecting it to $\Delta-1$ points between $u_\sigma$ and $u_1$; we finally add the leaves adjacent to $u_{\lfloor \frac{\sigma}{2} \rfloor +1}$ by connecting it to $\Delta-1$ points between $u_{\frac{\sigma}{2}}$ and $u_{\frac{\sigma}{2}+1}$ if $\sigma$ is even, or to $\Delta-1$ points  between $u_{\lfloor \frac{\sigma}{2} \rfloor +1}$ and $u_{\lfloor \frac{\sigma}{2} \rfloor +2}$ if $\sigma$ is odd.
The resulting drawing is called a \emph{zig-zag drawing} of $C$.

\smallskip

From now on, we assume that in a zig-zag drawing the points that represent vertices are equally spaced on the circle $\gamma$.
Let $\chi$ be the convex hull of the points representing the vertices of $C$ in $\Gamma$.
A zig-zag drawing has exactly two sides of $\chi$ that coincide with two edges of $C$; we call these two edges \emph{short edges} of $\Gamma$; each other side of $\chi$ is called a \emph{hole}. Denote by $v_1, v_2,\dots,v_n$ the vertices of $\Gamma$ according to the circular clockwise order they appear along $\chi$ with $v_1 \equiv u_1$; see \Cref{fi:zig-zag-drawing-b}. Notice that $(v_1,v_n)$ is a short edge and $v_n$ is the degree-1 vertex of this edge. 

Consider a straight line $s$ that intersects both short edges of $\Gamma$; line $s$ intersects all the  edges of the zig-zag drawing. Without loss of generality, assume that $s$ is horizontal and denote by $U$ the set of vertices that are above $s$ and by $L$ the set of vertices that are below $s$. The vertices in $U$ form the \emph{upper part} of $\Gamma$ and those in $L$ form the \emph{lower part} of $\Gamma$. Without loss of generality
 assume that $v_1$ is in the upper part (and therefore $v_n$ is in the lower part). It follows that each edge has the end-vertex with lower index in the upper part, and the end-vertex with higher index in the lower part.
 Hence the short edge different from $(v_1,v_n)$, which we denote as $(v_{r-1},v_r)$, is such that $v_{r-1}$ is in the upper part and $v_r$ is in the lower part. The first vertex of the upper part, i.e., vertex $v_1$, is called \emph{starting point} of $\Gamma$, while the first vertex of the lower part, i.e., vertex $v_r$, is called  \emph{ending point} of $\Gamma$.  We observe that $r=\frac{n}{2}+1$ if the number of vertices of the spine  $\sigma=\frac{n-2}{\Delta-1}$ is even, while $r=1+\frac{n-(\Delta-1)}{2}$ if $\sigma$ is odd. This can be written with a single formula as $r=1+\frac{n-(\Delta-1)(\sigma \mod 2)}{2}$.
 The two short edges separate two sets of consecutive holes, one completely contained in the upper part and one completely contained in the lower part; if $\sigma$ is even, these two sets have the same number of holes equal to $\frac{n-2}{2}$; if $\sigma$ is odd, then one of the two sets has $\frac{n-\Delta-1}{2}$ holes, while the other has $\frac{n+\Delta-3}{2}$. Note that the smaller set is in the upper part. 

 \smallskip

Let $\ell$ be a positive integer and let $\Gamma'$ be the drawing obtained by re-mapping vertex $v_i$ to the point\footnote{In a drawing in convex position the indices of the vertices are taken modulo $n$.}
representing $v_{i+\ell}$ in $\Gamma$. We say that $\Gamma'$ is the drawing obtained by \emph{rotating $\Gamma$ by $\ell$ steps.}
Note that the starting point of $\Gamma'$ is $v_{j}$ with $j=1+\ell$ and the ending point is $v_r$ with $r=1+\ell+\frac{n-(\Delta-1)(\sigma \mod 2)}{2}=j+\frac{n-(\Delta-1)(\sigma \mod 2)}{2}$.
The drawing in \Cref{fi:zig-zag-drawing-d} is the union of two zig-zag drawings $\Gamma_1$ and $\Gamma_2$, where $\Gamma_2$ is obtained by rotating $\Gamma_1$ by one step; the starting point of $\Gamma_1$ is $v_1$ while its ending point is $v_8$; the starting point of $\Gamma_2$ is $v_2$, while its ending point is $v_{9}$.

\begin{lemma}\label{le:rotation}
    Let $\Gamma_1$ be a zig-zag drawing of a $\Delta$-regular caterpillar $C$ with starting point $j_1$ and ending point $r_1$; let $\Gamma_2$ be a zig-zag drawing of $C$ with starting point~$j_2$. If $0 < j_2-j_1 < \frac{n-(\Delta-1)(\sigma \mod 2)}{2}$, where $\sigma$ is the number of spine vertices of $C$, then $\Gamma_1 \cup \Gamma_2$ has no multiple edges.
\end{lemma}

\begin{proof} 
We first observe that $\Gamma_2$ is obtained by rotating $\Gamma_1$ by $\ell$ steps, where $\ell=j_2-j_1$. Suppose that a multiple edge $(v_i,v_g)$, with $i<g$ exists in $\Gamma_1 \cup \Gamma_2$. This implies that in the drawing $\Gamma_1$ there must be an edge $(v_{i'},v_{g'})$ that, when rotated by $\ell$ steps, coincides with $(v_i,v_g)$. In other words, the two edges $(v_i,v_g)$ and $(v_{i'},v_{g'})$ must be such that: (i) $i' < i < r_1 \leq g < g'$; (ii) $g=i'+\ell$; (iii) the number $\alpha$ of vertices encountered between $v_i$ and $v_g$ when going clockwise from $v_i$ to $v_g$ is the same as the number of vertices encountered  when going clockwise from $v_{g'}$ to $v_{i'}$  (see Fig.~\ref{fi:overlapping}).
Denote by $\beta$ the number of vertices encountered when going clockwise from $v_{i'}$ to $v_i$, and by $\zeta$ the number of vertices encountered when going clockwise from $v_g$ to $v_{g'}$. We have $2\alpha+\beta+\zeta+4=n$. If $\sigma$ is even, then $\beta=\zeta$ (see \Cref{fi:overlapping-a}), which implies $\alpha+\beta+2=\frac{n}{2}$. Notice that $g=i'+\ell$ implies that $\ell=\beta+\alpha+2$ ($\ell$ is equal to the number of vertices encountered clockwise between $v_{i'}$ and $v_g$ plus one) and therefore $(v_{i'},v_{g'})$ can coincide with $(v_i,v_g)$ after a rotation of $\ell$ steps only if $\ell=\frac{n}{2}$ but, when $\sigma$ is even, we have $\ell=j_2-j_1<\frac{n}{2}$ and therefore a multiple edge cannot exist.

\begin{figure}[tbp]
		\centering
    	\subfigure[]{\label{fi:overlapping-a}\includegraphics[width=0.32\columnwidth, page=1]{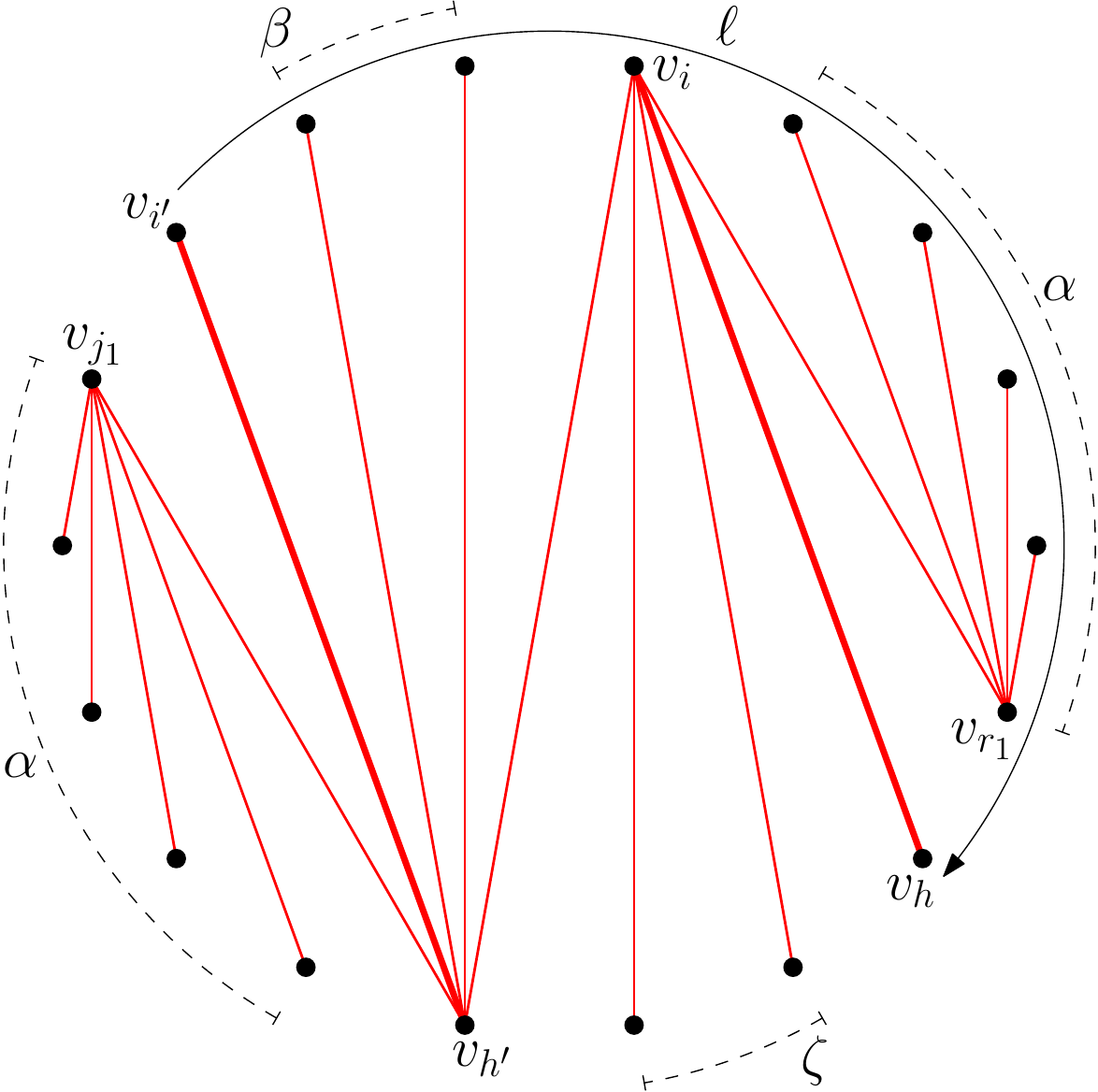}}
		\hfill
		\subfigure[]{\label{fi:overlapping-b}\includegraphics[width=0.32\columnwidth, page=2]{overlapping}}
		\hfill
		\subfigure[]{\label{fi:overlapping-c}\includegraphics[width=0.32\columnwidth, page=3]{overlapping}}
		\caption{\label{fi:overlapping} Illustration for the proof of \Cref{le:rotation}; (a) $\sigma$ even; (b)-(c) $\sigma$ odd.
		}
	\end{figure}

If $\sigma$ is odd, then $\beta - (\Delta-1) \leq \zeta \leq \beta+(\Delta-1)$ (see \Cref{fi:overlapping-b,fi:overlapping-c}) and therefore  $2\alpha+2\beta-(\Delta-1)+4 \leq 2\alpha+\beta+\zeta+4=n \leq 2\alpha+2\beta+(\Delta-1)+4$, which can be rewritten as $\frac{n-(\Delta-1)}{2} \leq \alpha+\beta+2 \leq \frac{n+(\Delta-1)}{2}$. It follows that, in order to have  $(v_{i'},v_{g'})$ and $(v_i,v_g)$ coincident after a rotation of $\ell$ steps, the value of $\ell$ must be such that $\frac{n-(\Delta-1)}{2} \leq \ell \leq \frac{n+(\Delta-1)}{2}$; but, when $\sigma$ is odd, we have $\ell=j_2-j_1<\frac{n-(\Delta-1)}{2}$ and therefore a multiple edge cannot exist.
\end{proof} 

We conclude this section by computing the maximum number of crossings per edge in the union of two zig-zag drawings without overlapping edges.  We state this lemma in general terms assuming that the two $\Delta$-regular caterpillars can have different vertex degrees, as we are going to use the lemma to establish upper bounds on $k$ both for $k$-planar $h$-placements and for $k$-planar $h$-packings (Section~\ref{se:packing}).

Let $\Gamma$ be a union of a set of zig-zag drawings. To ease the description that follows, we regard $\Gamma$ as a sub-drawing of a straight-line drawing of $K_n$ whose vertices coincide with those of $\Gamma$ (and therefore are equally spaced along a circle). In particular, for each vertex $v_j$, we denote by $e_{j,0}, e_{j,1},\dots, e_{j,n-2}$ the edges incident to $v_j$ in $K_n$ according to the circular counterclockwise order around $v_j$ starting from $e_{j,0}=(v_j,v_{j-1})$. Each of the zig-zag drawings that form $\Gamma$ contains a subset of these edges and $\Gamma$ is a valid packing if there is no edge that belongs to two different zig-zag drawings in the set whose union is $\Gamma$.

We denote by $\mathcal S_n$ the (circular) sequence of slopes  $s_i = i \cdot \frac{\pi}{n}$, for $i=0,1,\dots,n-1$; refer to \Cref{fi:slopes}. Notice that, without loss of generality, we can assume that the convex hull of $\Gamma$ has a side with slope $s_0$ and, as a consequence, every edge of $\Gamma$ has a slope in the set $\mathcal S_n$. Let $v_j$ be a vertex; if the slope of $e_{j,0}$ is $s_{i_j}$, then the slope of $e_{j,p}$ is $s_{{i_j}+p}$ (with indices taken modulo $n$); in other words, the edges incident to each vertex have slopes that form a sub-sequence of $n-1$ consecutive elements of $\mathcal S_n$; we denote such a sequence as $\psi(i_j)$, where $i_j$ indicates that the first element of $\psi(i_j)$ is $s_{i_j}$. We say that $v_j$ \emph{uses} the sequence $\psi(i_j)$. If we consider two different vertices $v_{j}$ and $v_{j+p}$ and $v_j$ uses the sequence $\psi(i_j)$, then $v_{j+p}$ uses the sequence $\psi(i_j-2p)$ (with indices taken modulo $n$); in other words, the sequence used by a vertex shifts clockwise by two elements moving to the next vertex.

\begin{figure}[tb]
	\centering
	\includegraphics[width=0.7\textwidth, page=3]{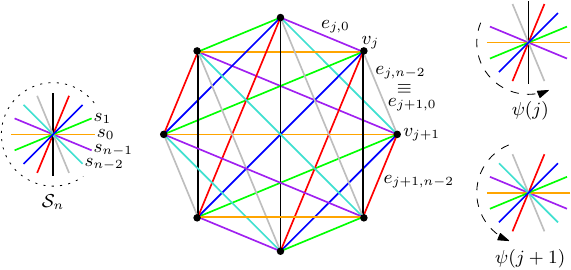}
	\caption{Illustration for the definition of slopes.}\label{fi:slopes} 
\end{figure}

\begin{lemma}\label{le:crossings}
    Let $C_1$ be an $n$-vertex $\Delta_1$-regular caterpillar and let $C_2$ be an $n$-vertex $\Delta_2$-regular caterpillar with $\Delta_i\leq n-2$ (for $i=1,2$). Let $\Gamma_1$ be a zig-zag drawing of $C_1$ with starting point $v_{j_1}$ and let $\Gamma_2$ be a zig-zag drawing of $C_2$ with starting point $v_{j_2}$ with $0 < j_2-j_1 < \frac{n}{2}$. If $\Gamma_1 \cup \Gamma_2$ has no multiple edges, then any edge of $\Gamma_1 \cup \Gamma_2$  is crossed at most $2(\Delta_1+\Delta_2)+4(j_2-j_1)$ times.
\end{lemma}
\begin{proof}
    We first observe that the edges of a zig-zag drawing of a $\Delta$-regular caterpillar are all drawn as segments whose slope belongs to a set of $\Delta$ slopes. In particular, for every spine vertex $v$, the edges incident to $v$ are drawn using all these $\Delta$ slopes. 

    Consider the starting vertex $v_{j_1}$ of $\Gamma_1$; the edges incident to $v_{j_1}$ are drawn with the first $\Delta_1$ slopes of $\psi(i_{j_1})$. Analogously, the edges incident to the starting vertex $v_{j_2}$ of $\Gamma_2$ are drawn with the first $\Delta_2$ slopes of $\psi(i_{j_2})$. The sequence $\psi(i_{j_2})$ is shifted clockwise by $2(j_2-j_1)$ units with respect to $\psi(i_{j_1})$. On the other hand, since $j_2-j_1 < \frac{n}{2}$, the first slope of $\psi(i_{j_2})$ is distinct from the first slope of $\psi(i_{j_1})$.

    \begin{figure}[tbp]
		\centering
    	\subfigure[]{\label{fi:crossings-4cases-a}\includegraphics[width=0.45\columnwidth, page=1]{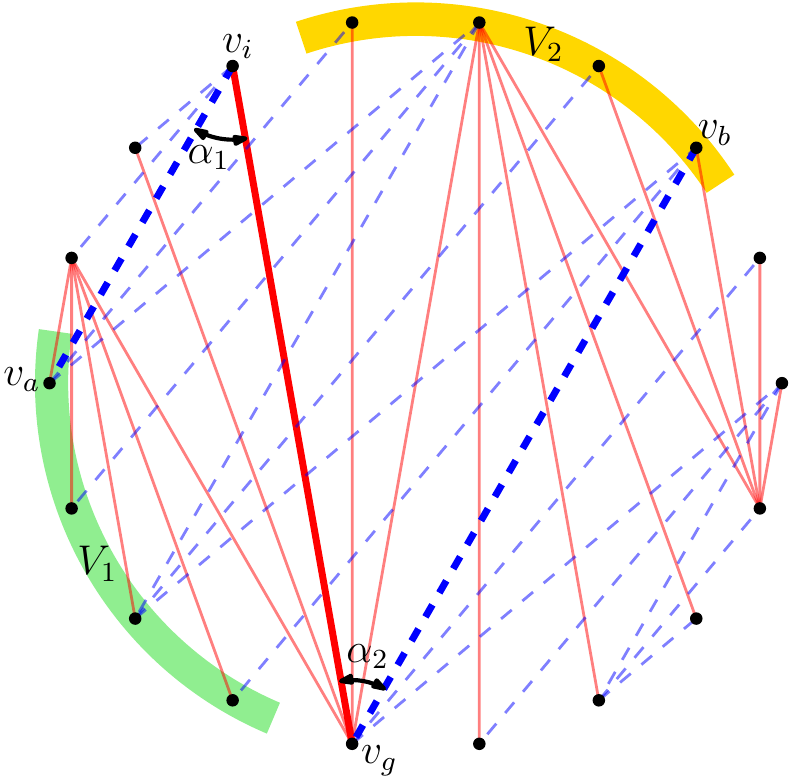}}
		\hfill
		\subfigure[]{\label{fi:crossings-4cases-b}\includegraphics[width=0.45\columnwidth, page=2]{crossings-4cases}}
		\hfill
		\subfigure[]{\label{fi:crossings-4cases-c}\includegraphics[width=0.45\columnwidth, page=3]{crossings-4cases}}
		\hfill
		\subfigure[]{\label{fi:crossings-4cases-d}\includegraphics[width=0.45\columnwidth, page=4]{crossings-4cases}}\caption{\label{fi:crossings-4cases} Illustration for the proof of \Cref{le:crossings}. The edges of $\Gamma_2$ are dashed. 
		}
    \end{figure}
    
    Let $e=(v_i,v_g)$ be an edge of $\Gamma_1$. We now prove that the number of crossings along $e$ is at most the one given in the statement. Let $e_1=(v_i,v_a)$ be the edge of $\Gamma_2$ incident to $v_i$ that forms the smallest angle with $e$; analogously, let $e_2=(v_g,v_b)$ be the edge of $\Gamma_2$ incident to $v_g$ that forms the smallest angle with $e$. Notice that, in principle there are four possible clockwise orders of $v_i$, $v_a$, $v_g$, and $v_b$
    (see cases (a)--(d) in \Cref{fi:crossings-4cases} for an illustration).  However the case (b) cannot happen. Namely, in case (b) the slopes used to draw the edges of $\Gamma_2$ would be shifted counterclockwise with respect to those used to represent the edges of $\Gamma_1$; but, as observed above, the slopes used by $\Gamma_2$ are shifted clockwise with respect to those used by $\Gamma_1$.  
    
    Let $\alpha_1$ be the angle between $e$ and $e_1$ and let $\alpha_2$ be the angle between $e$ and $e_2$. Let $V_1$ be the set of vertices seen by the angle $\alpha_1$ including $v_a$ and excluding $v_g$; analogously let $V_2$ be the set of vertices seen by the angle $\alpha_2$ including $v_b$ and excluding $v_i$. In each of the three cases (a), (c), and (d), at least one of $\alpha_1$ and $\alpha_2$ is such that $e$ sweeps the angle moving clockwise. Let $\alpha_l$ with $l \in \{1,2\}$ be the angle that satisfies this condition.  In particular, for case (a) $\alpha_l$ can be both $\alpha_1$ or $\alpha_2$, in case (c) $\alpha_l$ is $\alpha_2$ and in case (d) $\alpha_l$ is $\alpha_1$ (see \Cref{fi:crossings-4cases}). Every edge that crosses $e$ has an end-vertex in $V_1$ and one end-vertex in $V_2$. To count the number of such edges (and therefore the number of crossings along $e$), we evaluate $|V_l|$. The value of $|V_l|$ is at most the number of slopes of $\mathcal S_n$ that are encountered in counterclockwise order between the slope $s \in \mathcal{S}_n$ of $e_l$ and the slope $s' \in \mathcal{S}_n$ of $e$. In particular, in case (a) $|V_l|$ is exactly this number, while in case (c) and (d) $|V_l|$ is less than this number.    
    The slope $s'$ is at most the last slope used by $\Gamma_1$, which is $s_{p}$ with $p=j_1+\Delta_1$, while the slope $s$ is at least the first slope used by $\Gamma_2$, which is $s_{q}$ with $q=j_1-2(j_2-j_1)$. Thus, the number of slopes between $s'$ (included) and $s$ (excluded) is at most $p-q=\Delta_1+2(j_2-j_1)$. Hence $|V_l| \leq \Delta_1+2(j_2-j_1)$. 
    
    We call a \emph{block} a subset of consecutive vertices of $V_l$ starting with a spine vertex and containing all the leaves that follow that spine vertex. The number of edges of $\Gamma_2$ incident to the vertices of a block is $2(\Delta_2-1)$ (since $\Delta_2$ edges are incident to the spine vertex and $\Delta_2-2$ is the number of leaves). The number of blocks in $V_l$ is $\left \lceil \frac{|V_l|}{(\Delta_2-1)} \right\rceil$. It follows that the number of crossings $\chi_e$ along $e$ is at most $\left \lceil \frac{|V_l|}{(\Delta_2-1)} \right\rceil 2(\Delta_2-1)$ which is less than $2(|V_l|+\Delta_2)$. Since $|V_l|\leq \Delta
    _1+2(j_2-j_1)$, we have $\chi_e \leq 2(\Delta_1+\Delta_2)+4(j_2-j_1)$, which concludes the proof in the case when $e$ belongs to $\Gamma_1$. The case when the edge $e$ belongs to $\Gamma_2$ is analogous. In particular, when $e$ belongs to $\Gamma_2$, the cases (b), (c), and (d) apply, while case (a)~does~not~happen.  
\end{proof}

\subsection{Characterization}\label{sse:characterization}

We are now ready to characterize the $\Delta$-regular caterpillars that admit an $h$-placement. 

\begin{theorem}\label{th:characterization}
   Let $C$ be a $\Delta$-regular caterpillar with $n$ vertices. An $h$-placement of $C$ exists if and only if: (i) $\Delta \leq n-h$; and (ii) $n \geq 2h+(\Delta-1)\cdot(\sigma \mod 2)$, where $\sigma$ is the number of spine vertices of $C$. Further, if an $h$-placement exists, there exists one that is $k$-planar for $k \in O(\Delta h+h^2)$.
\end{theorem}
\begin{proof}
We first prove the sufficient condition. 
Let $C_1, C_2, \dots, C_h$ be the $h$ caterpillars and assume that $n\geq 2h+(\Delta-1)(\sigma \mod 2)$. 
We compute an $h$-placement of $C_1, C_2,\dots,C_h$ starting from a zig-zag drawing $\Gamma_1$ of $C_1$ and obtaining the drawing $\Gamma_i$ of $C_i$ by rotating $\Gamma_1$ by $i-1$ steps, for $i=2,3,\dots, h$. 

Notice that, when the number of spine vertices $\sigma$ of each $C_i$ is even, $h \leq \frac{n}{2}$ and therefore each $\Gamma_i$ is rotated by less than $\frac{n}{2}$ steps; when $\sigma$ is odd $h \leq \frac{n-(\Delta-1)}{2}$ and each $\Gamma_i$ is rotated by less than $\frac{n-(\Delta-1)}{2}$ steps. In both cases, each pair of drawings $\Gamma_i$ and $\Gamma_j$ satisfies the conditions of \Cref{le:rotation} and therefore there are no multiple edges, that is, the union of all $\Gamma_i$ is a valid $h$-placement of $C_1, C_2, \dots, C_h$. 

We now prove the necessary condition. If $\sigma$ is even, then conditions (i) and (ii) are necessary by \Cref{prop:necessary}. Hence, consider the case when $\sigma$ is odd. Condition (i) is necessary by \Cref{prop:necessary}. Assume, by contradiction, that (ii) is not necessary, i.e., there exists an $h$-placement of $h$ caterpillars $C_1, C_2, \dots, C_h$ such that $n<2h+(\Delta-1)$. Since $C_1, C_2, \dots, C_h$ admit an $h$-placement, by \Cref{prop:necessary} $n$ must be at least $2h$. Thus, it would be $2h \leq n < 2h+(\Delta-1)$; in other words, $n=2h+\alpha$ with $0 \leq \alpha \leq \Delta-2$. 

Let $G$ be the host graph of the $h$-placement and let $v$ be the vertex of $G$ to which the largest number of spine vertices of $C_1, C_2, \dots, C_h$ is mapped. Let $\beta$ be the number of spine vertices that are mapped to $v$. There are other $h-\beta$ leaf vertices that are mapped to $v$ (because one vertex per caterpillar has to be mapped on each vertex of $G$). The degree of $v$ in $G$ is at most $n-1$ and each of the spine vertices mapped to $v$ has degree $\Delta$. Hence, the $\beta$ spine vertices mapped to $v$ have degree $\beta \Delta$ in total. Vertex $v$ can have  at most other $n-1-\beta\Delta$ edges and therefore it must be $n-1-\beta\Delta \geq h-\beta$, i.e., $\beta \leq \frac{n-1-h}{\Delta-1}$.
On the other hand, there are $\sigma h$ spine vertices in total and, since $G$ has $n$ vertices, there are at least $\lceil \frac{\sigma h}{n} \rceil$ spine vertices mapped to $v$, i.e.,
$\beta \geq \lceil \frac{\sigma h}{n} \rceil$.
Putting together the two conditions on $\beta$ we obtain:

\begin{equation*}\label{eq:x}
    \left\lceil \frac{\sigma h}{n} \right\rceil \leq \beta \leq \frac{n-1-h}{\Delta-1}. 
\end{equation*}

Since $n=2h+\alpha$, we have $h=\frac{n-\alpha}{2}$; replacing $h$ in \Cref{eq:x}, we obtain:
\begin{equation*}\label{eq:x1}
    \left\lceil \frac{\sigma}{2} - \frac{\sigma\alpha}{2n}\right \rceil \leq \beta \leq \frac{n+\alpha-2}{2(\Delta-1)}. 
\end{equation*}
Since $n=\sigma(\Delta-1)+2$, we have:
\begin{equation}\label{eq:x2}
    \left \lceil \frac{\sigma}{2} - \frac{\sigma\alpha}{2(\sigma(\Delta-1)+2)}\right \rceil \leq \beta \leq \frac{\sigma(\Delta-1)+\alpha}{2(\Delta-1)}. 
\end{equation}
\Cref{eq:x2} implies that:
\begin{equation}\label{eq:x3}
    \left \lceil \frac{\sigma}{2} - \frac{\alpha}{2(\Delta-1)+\frac{4}{\sigma}}\right \rceil \leq \frac{\sigma}{2}+\frac{\alpha}{2(\Delta-1)}. 
\end{equation}

We now prove that \Cref{eq:x3} cannot be satisfied. Since $\sigma$ is odd, it is $\sigma=2i+1$ for some $i \in \mathbb{N}$, and thus:
\begin{equation}\label{eq:x4}
    \left\lceil i + \frac{1}{2}-\zeta \right\rceil\leq k + \frac{1}{2}+\zeta', 
\end{equation}
with $\zeta=\frac{\alpha}{2(\Delta-1)+\frac{4}{\sigma}}$ and $\zeta'=\frac{\alpha}{2(\Delta-1)}$. We have $\zeta < \zeta'$ and we prove that $\zeta' < \frac{1}{2}$:
\begin{equation*}
    \zeta'=\frac{\alpha}{2(\Delta-1)}\leq \frac{\Delta-2}{2(\Delta-1)}<\frac{\Delta-1}{2(\Delta-1)}=\frac{1}{2}.
\end{equation*}
The first term of \Cref{eq:x4} is $i+1$ because $0<\frac{1}{2}-\zeta<1$; the second term is less than $i+1$ because $0<\frac{1}{2}+\zeta'<1$. It follows that \Cref{eq:x4} does not hold and therefore \Cref{eq:x3} does not hold.

\smallskip

We now prove the bound on the number of crossings along an edge. We consider an edge of $\Gamma_1$; the number of crossings along an edge of the drawing of another caterpillar is bounded by the same number. Let $e$ be an edge of the drawing $\Gamma_1$. By \Cref{le:crossings}, the number of crossings $\chi_e$ along $e$ due to the edges of another drawing $\Gamma_l$ (with $2 \leq l \leq h$) is at most $2(\Delta_1+\Delta_l)+4(j_l-j_1)$. Summing over all drawings distinct from $\Gamma_1$, we obtain $\chi_e \leq \sum_{l=2}^h(2(\Delta_1+\Delta_l)+4(j_l-j_1))$. Considering that $\Delta_l=\Delta$ for every $l$ and that $j_l-j_1=l-1$, we have 
\begin{equation}\label{eq:crossings}
 \chi_e \leq \sum_{l=2}^h(4\Delta+4(l-1))\leq (4\Delta-2)h+2h^2-4\Delta.   
\end{equation}

\end{proof}

\begin{figure}[tbp]
	\centering
    \subfigure[]{\label{fi:zig-zag-drawing-inner}\includegraphics[width=0.44\columnwidth, page=5]{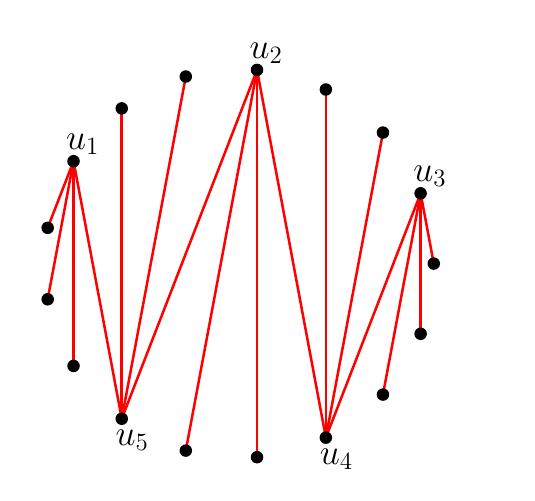}}
	\hfil
	  \subfigure[]{\label{fi:zig-zag-drawing-outer}\includegraphics[width=0.44\columnwidth, page=3]{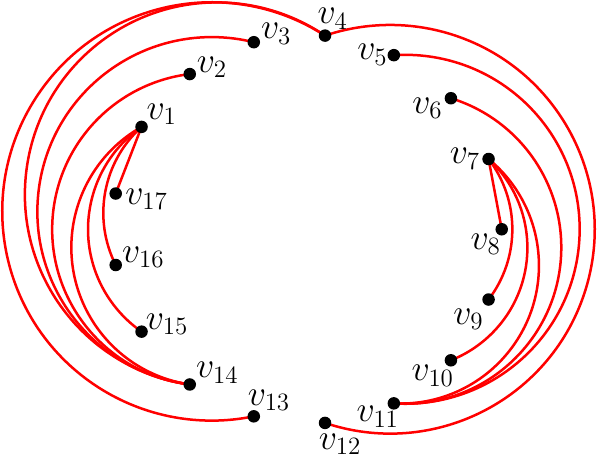}}
	\caption{\label{fi:zig-zag-drawing-inner-outer} (a) An inner zig-zag drawing and (d) an outer zig-zag drawing of a $4$-regular caterpillar.
	}
\end{figure}

We conclude by observing that the number of crossings given by \Cref{eq:crossings} can be reduced, although not asymptotically. A zig-zag drawing can be embedded inside the circle (see \Cref{fi:zig-zag-drawing-inner}) or outside the circle (see \Cref{fi:zig-zag-drawing-outer}). Thus, the number given by \Cref{eq:crossings} can be halved by embedding half of the caterpillars inside the circle and the other half outside the circle.

\section{\texorpdfstring{$h$}--packing of \texorpdfstring{$\Delta$}--regular Caterpillars in \texorpdfstring{$k$}--planar Graphs}\label{se:packing}

In this section we study $h$-packings of $h$ $\Delta_1$-, $\Delta_2$-, $\dots$, $\Delta_h$-regular caterpillars
 such that the degree $\Delta_i$ and the degree $\Delta_j$ of the spine vertices can differ from one caterpillar to another, for  $1 \leq i,j \leq h$, $i\neq j$. 
 
 \begin{lemma}\label{le:rotation-3}
    Let $C_1$ be an $n$-vertex $\Delta_1$-regular caterpillar and let $C_2$ be an $n$-vertex $\Delta_2$-regular caterpillar such that $\Delta_1 > \Delta_2$ and $\Delta_1 \leq n-2$. Let $\Gamma_1$ be a zig-zag drawing of $C_1$ with starting point $v_{j_1}$ and ending point $v_{r_1}$, and let $\Gamma_2$ be a zig-zag drawing of $C_2$ with starting point $v_{j_2}$ and ending point $v_{r_2}$. If $\frac{\Delta_2}{2} \leq j_2-j_1 < \frac{n-(\Delta_1-1)}{2}$, 
    then $\Gamma_1 \cup \Gamma_2$ has no multiple edges.
\end{lemma}
\begin{proof}
As described in the proof of \Cref{le:crossings}, the edges of a zig-zag drawing of a $\Delta$-regular caterpillar are all drawn as segments whose slope belongs to a set of $\Delta$ slopes. In particular, for every spine vertex $v$, the edges incident to $v$ are drawn using all these $\Delta$ slopes. Based on this observation, we show that the $\Delta_1$ slopes used to represent the edges of $\Gamma_1$ are distinct from the $\Delta_2$ slopes used to represent the edges of $\Gamma_2$. We use the same notation used in \Cref{th:characterization}.

Consider the staring vertex $v_{j_1}$ of $\Gamma_1$; the edges incident to $v_{j_1}$ are drawn with the first $\Delta_1$ slopes of $\psi(i_{j_1})$. Analogously, the edges incident to the starting vertex $v_{j_2}$ of $\Gamma_2$ are drawn with the first $\Delta_2$ slopes of $\psi(i_{j_2})$. Since $j_2-j_1 \geq \frac{\Delta_2}{2}$, the sequence $\psi(i_{j_2})$ is shifted clockwise by $\Delta_2$ units with respect to $\psi(i_{j_1})$. On the other hand, since $j_2-j_1 \leq \frac{n-(\Delta_1-1)}{2}$, the sequence of the first $\Delta_2$ slopes of $\psi(i_{j_2})$ does not overlap with the first $\Delta_1$ slopes of $\psi(i_{j_1})$, which concludes the proof.\end{proof}

\begin{theorem}\label{th:general-caterpillars}
    Let $C_1, C_2, \dots, C_h$ be $h$ caterpillars such that $C_i$ is $\Delta_i$-regular, for $1 \le i\le h$, and $\Delta_{h}\leq\Delta_{h-1}\leq \dots \leq \Delta_1 \leq n-h$. If $\sum_{i=1}^h \Delta_i \le n-1$ and $\sum_{i=2}^h \left\lceil\frac{\Delta_i}{2}\right\rceil< \frac{n-(\Delta_1-1)}{2}$, then there exists a $k$-planar packing with $k \in O(\Delta_1 h^2)$.
\end{theorem}
\begin{proof} We compute a zig-zag drawing of $C_1$ with starting point $v_{j_1}$, with $j_1=1$; for each $C_i$, with $2 \leq i \leq h$, we compute a zig-zag drawing $\Gamma_i$ with starting vertex $v_{j_i}$ where $j_i=j_{i-1}+\left\lceil \frac{\Delta_i}{2} \right \rceil$. Notice that, each vertex $v$ of $\Gamma_1 \cup \Gamma_2 \cup \dots \cup \Gamma_h$ has degree at most $n-1$; namely $\sum_{i=1}^{h}\deg_{C_i}(v)\leq \sum_{i=1}^h \Delta_i \le n-1$. Moreover, given two caterpillars $C_{i}$ and $C_{i'}$ with $1 \leq i < i' \le h$, we have that: (i) $j_{i'}-j_i \geq j_{i'}-j_{i'-1}=\left\lceil\frac{\Delta_{i'}}{2}\right\rceil$; and (ii) $j_{i'}-j_i\leq j_h-j_1=\sum_{i=2}^{h}\lceil\frac{\Delta_i}{2}\rceil$, which gives $j_{i'}-j_i < \frac{n-(\Delta_1-1)}{2}<\frac{n-(\Delta_{i}-1)}{2}$. Putting together (i) and (ii), we obtain  $\frac{\Delta_{i'}}{2}<j_{i'}-j_i<\frac{n-(\Delta_i-1)}{2}$. Hence, \Cref{le:rotation-3} holds for every pair of caterpillars and the union of all the zig-zag drawings $\Gamma_1, \Gamma_2, \dots, \Gamma_h$ is a valid packing of $C_1, C_2, \dots, C_h$.

We now prove the bound on the number of crossings along an edge. We consider an edge of $\Gamma_1$; the number of crossings along an edge of another drawing is bounded by the same number. Let $e$ be an edge of the drawing $\Gamma_1$. By \Cref{le:crossings}, the number of crossings $\chi_e$ along $e$ due to the edges of another drawing $\Gamma_l$ (with $2 \leq l \leq h$) is at most $2(\Delta_1+\Delta_l)+4(j_l-j_1)$. Summing over all drawings distinct form $\Gamma_1$, we obtain $\chi_e \leq \sum_{l=2}^h(2(\Delta_1+\Delta_l)+4(j_l-j_1))$. Considering that $j_l \geq j_{l-1}+\left\lceil\frac{\Delta_i}{2}\right\rceil$, we obtain that $j_l-j_1=\sum_{i=2}^{l}\left\lceil\frac{\Delta_i}{2}\right\rceil$. Since $\Delta_l \leq \Delta_1$ for every $2 \leq l \leq h$, we have $j_l-j_1 \leq(l-1)(\frac{\Delta_1}{2}+1)$. Therefore, we obtain $\chi_e \leq \sum_{l=2}^h(4\Delta_1+4(l-1)(\frac{\Delta_1}{2}+1))\leq (\Delta_1+2)h^2+4\Delta_1(h-1)$.
\end{proof}

We now consider a special case of packing a set of $h$ $\Delta_1$-, $\Delta_2$-, $\dots$, $\Delta_h$-regular caterpillars where, for each $\Delta_i$ $(1\leq i \leq h$), we have that $\Delta_i-1$ is a multiple of $\Delta_{i+1}-1$. In this case, we show that the sufficient conditions of \Cref{th:general-caterpillars} can be relaxed.  For example, consider the packing of a $17$-regular caterpillar and two $9$-regular caterpillars, each having $34$ vertices. These three caterpillars do not satisfy the sufficient condition of \Cref{th:general-caterpillars}. However, a $k$-planar packing of these caterpillars is possible, as proven in \Cref{th:upper-regular-2}.
We start with the following property, which immediately follows from the construction of a zig-zag drawing (see also \Cref{fi:spine-indices} for an illustration).

\begin{property}\label{prop:spine-indices}
Let $\Gamma$ be a zig-zag drawing of a $\Delta$-regular caterpillar with starting vertex $v_j$ and ending vertex $v_r$. If $v_i$ is a spine vertex in the upper part of $\Gamma$, then $i=j+c(\Delta-1)$ for some $c \in \mathbb{N}$; if $v_g$ is a spine vertex in the lower part of $\Gamma$, then $g=r+d(\Delta-1)$ for some $d \in \mathbb{N}$. Moreover, if $v_i$ and $v_g$ are adjacent then either $c+d=\left \lceil \frac{\sigma}{2} \right \rceil-1$ or $c+d=\left \lceil \frac{\sigma}{2} \right \rceil$, where $\sigma$ is the number of spine
vertices of $\Gamma$.
\end{property}

\begin{figure}[tb]
	\centering
	\includegraphics[width=0.44\textwidth, page=1]{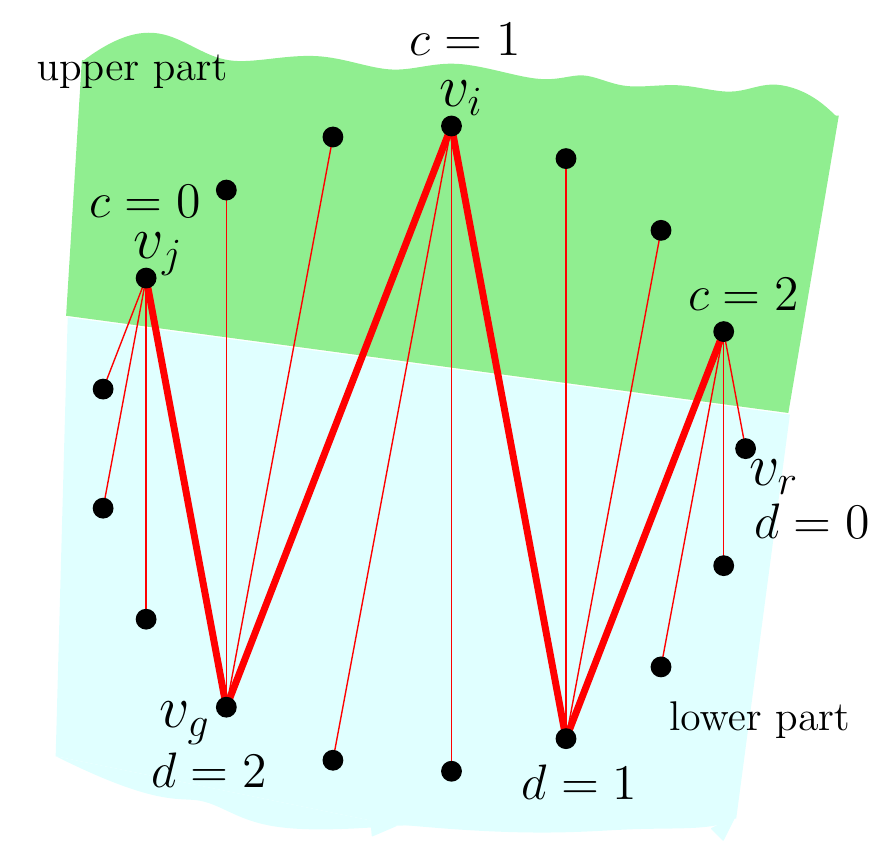}
	\caption{Illustration for \Cref{prop:spine-indices}; $\sigma=5$. For each spine vertex, $c$ and $d$ are shown. Considering adjacent spine vertices, the sum of $c$ and $d$ is 2 or 3.}\label{fi:spine-indices} 
\end{figure}

\Cref{prop:spine-indices} is extensively used in the proof of the following lemma.

\begin{lemma}\label{le:no-multiple-edges-multiple}
    Let $C_1$ be an $n$-vertex $\Delta_1$-regular caterpillar and let $C_2$ be an $n$-vertex $\Delta_2$-regular caterpillar such that $\Delta_1-1=q(\Delta_2-1)$, for some $q \in \mathbb{N}^+$ and $\Delta_i\leq n-2$ (for $i=1,2$). Let $\Gamma_1$ be a zig-zag drawing of $C_1$ with starting point $v_{j_1}$ and ending point $v_{r_1}$, and let $\Gamma_2$ be a zig-zag drawing of $C_2$ with starting point $v_{j_2}$ and ending point $v_{r_2}$. If $0 < j_2-j_1 < \frac{n-(\Delta_1-1)}{2}$, then $\Gamma_1 \cup \Gamma_2$ has no multiple edges.
\end{lemma}
\begin{proof}
    Let $(v_{i_1},v_{g_1})$ be an edge of $\Gamma_1$ and $(v_{i_2},v_{g_2})$ be an edge of $\Gamma_2$. Assume that $v_{i_1}$ belongs to the upper part of $\Gamma_1$ and $v_{i_2}$ belongs to the upper part of $\Gamma_2$. Note that this implies that $v_{g_1}$ belongs to the lower part of $\Gamma_1$ and $v_{g_2}$ belongs to the lower part of $\Gamma_2$. We prove that $(v_{i_1},v_{g_1})$ and $(v_{i_2},v_{g_2})$ do not coincide.

    We first show that it does not happen that $v_{i_1}$ coincides with $v_{i_2}$ and $v_{g_1}$ coincides with $v_{g_2}$. We then show that it does not happen that $v_{i_1}$ coincides with $v_{g_2}$ and $v_{g_1}$ coincides with $v_{i_2}$. In the rest of the proof we will express the four indices $i_1$, $i_2$, $g_1$ and $g_2$ in terms of the values $j_1$, $j_2$, $r_1$ and $r_2$, according to \Cref{prop:spine-indices}. Without loss of generality, we can assume that $r_2 \leq n$ and $j_1 \geq 1$, i.e., that the vertices $v_{r_2}$, $v_n$, $v_1$, and $v_{j_1}$ appear in this clockwise order, with $v_{r_2}$ and $v_n$ possibly coincident and with $v_1$ and $v_{j_1}$ possibly coincident. With these assumptions, we have $j_1 < j_2 < r_1 < r_2$ and $v_{i_1}$ can coincide with $v_{g_2}$ only if $i_1=g_2-n$, i.e., only if the value of $g_2$ is greater than $n$ and coincides with $i_1$ modulo $n$. Thus, while assuming that $v_{i_1}$ coincides with $v_{i_2}$ implies that $i_1=i_2$, assuming that $v_{i_1}$ coincides with $v_{g_2}$ implies that $i_1=g_2-n$. 

    \medskip \noindent \textbf{Case 1: It does not happen that $v_{i_1}$ coincides with $v_{i_2}$ and $v_{g_1}$ coincides with $v_{g_2}$.}    
    
    At least one vertex per edge is a spine vertex. We distinguish four sub-cases depending on which vertex is a spine vertex for each edge. Since all the cases are very similar, we give here only the first case and the others can be found in the appendix.

    \medskip \noindent \textbf{Case 1.a:  $v_{i_1}$ and $v_{i_2}$ are spine vertices.}
    By \Cref{prop:spine-indices} we have, for some $c_1,c_2 \in \mathbb{N}$:
    \begin{equation}\label{eq:c1-i1}
        i_1=j_1+c_1(\Delta_1-1)=j_1+qc_1(\Delta_2-1).
    \end{equation}
    and 
    \begin{equation}\label{eq:c1-i2}
        i_2=j_2+c_2(\Delta_2-1).
    \end{equation}
    
    If $v_{i_1}$ coincides with $v_{i_2}$, we have 
    $i_1=i_2$; from \Cref{eq:c1-i1} and \Cref{eq:c1-i2} we obtain:
    \begin{equation}\label{eq:c1-i1-A}
        j_2-j_1 = (qc_1-c_2)(\Delta_2-1).
    \end{equation}
    
    \noindent Concerning $v_{g_1}$ and $v_{g_2}$, we have:
    \begin{equation*}\label{eq:c1-g1}
    g_1^{m} \leq g_1 \leq g_1^{M};
    \end{equation*}
    \begin{equation*}\label{eq:c1-g2}
    g_2^{m} \leq g_2 \leq g_2^{M}.
    \end{equation*}
   
    \noindent with $g_l^{m}=r_l+d_l(\Delta_l-1)$, $g_l^{M}=r_l+(d_l+1)(\Delta_l-1)$ for some $d_l \in \mathbb{N}$ such that $c_l+d_l=\left\lceil \frac{\sigma_l}{2}\right \rceil-1$, where $\sigma_l$ is the number of spine vertices of $C_l$, for $l=1,2$.     

    \noindent We prove that $g_1^{M} < g_2^{m}$, which implies $g_1 \neq g_2$. To have $g_1^{M} < g_2^{m}$ it must be:
    \begin{align*}
        &r_1+(d_1+1)(\Delta_1-1) < r_2+d_2(\Delta_2-1)\\
        &r_1+q(d_1+1)(\Delta_2-1) < r_2+d_2(\Delta_2-1)\\
        &r_2-r_1 > (qd_1+q-d_2) (\Delta_2-1) \numberthis \label{eq:c1-i1-B}.
    \end{align*}
    \noindent Since $r_l=j_l+\frac{n-(\Delta_l-1)(\sigma_l \mod 2)}{2}$, for $l=1,2$, \Cref{eq:c1-i1-B} can be rewritten as:
    \begin{equation}\label{eq:c1-i1-B'}
        j_2-j_1>\left( (qd_1+q-d_2) +\frac{(\sigma_2 \mod 2)}{2}-\frac{q(\sigma_1 \mod 2) }{2}\right)(\Delta_2-1).
    \end{equation}
\noindent Combining \Cref{eq:c1-i1-A} and \Cref{eq:c1-i1-B'} we obtain:
    \begin{equation*}
        qc_1-c_2> (qd_1+q-d_2) +\frac{(\sigma_2 \mod 2)}{2}-\frac{q(\sigma_1 \mod 2) }{2}.
    \end{equation*}

\noindent Since $c_l+d_l=\left\lceil \frac{\sigma_l}{2}\right \rceil-1$, we have $d_l= \frac{\sigma_l+1(\sigma_l \mod 2)}{2}-c_l-1$, for $l=1,2$; replacing $d_1$ and $d_2$ in the previous equation, we obtain:
    \begin{align*}
        &qc_1-c_2> \frac{q\sigma_1}{2}+\frac{q(\sigma_1 \mod 2)}{2}-qc_1-q+q-\frac{\sigma_2}{2}-\frac{1(\sigma_2 \mod 2)}{2}+c_2+\\
        &+1+\frac{\sigma_2 \mod 2}{2}-\frac{q(\sigma_1 \mod 2)}{2}
    \end{align*}
\noindent which, considering that $\sigma_2=\frac{n-2}{\Delta_2-1}=\frac{q(n-2)}{\Delta_1-1}=q\sigma_1$, implies:
    \begin{equation}\label{eq:c1-i1-C}
        qc_1-c_2> \frac{1}{2}
    \end{equation}

\noindent In summary, to have $g_1^{M} < g_2^{m}$ \Cref{eq:c1-i1-C} must hold. On the other hand, from \Cref{eq:c1-i1-A} and from the hypothesis that $j_2-j_1>0$ we obtain $(qc_1-c_2)(\Delta_2-1)>0$ which, since $(\Delta_2-1)>0$, implies $qc_1-c_2>0$ and, since $qc_1-c_2$ is integer, can be rewritten as $qc_1-c_2\geq 1$. This implies that \Cref{eq:c1-i1-C} holds and therefore that $g_1^{M} < g_2^{m}$ and $g_1 \neq g_2$.

\medskip \noindent \textbf{Case 2: It does not happen that $v_{i_1}$ coincides with $v_{g_2}$ and $v_{g_1}$ coincides with $v_{i_2}$.}
    
    Also in this case we distinguish four sub-cases depending on which vertex is a spine vertex for each edge. As in \textbf{Case 1}, we give here only the first sub-case, while the others can be found in the appendix.
    
    \medskip \noindent \textbf{Case 2.a: $v_{i_1}$ and $v_{i_2}$ are spine vertices.} Since $v_{g_2}$ is a vertex in the lower part of $\Gamma_2$, it must be $g_2=r_2+d_2(\Delta_2-1)+\alpha_2$, for some $\alpha_2$ such that  $0 \leq \alpha_2 < \Delta_2-1$. If $v_{g_2}$ coincides with $v_{i_1}$, as explained above, it must be $i_1=g_2-n$. Combining the expression of $g_2$ with \Cref{eq:c1-i1} we obtain:
    \begin{equation}\label{eq:c2-i1-A}
        r_2-j_1 = (qc_1-d_2)(\Delta_2-1)-\alpha_2+n.
    \end{equation}
    
    \noindent Concerning $v_{g_1}$, we have:
    \begin{equation*}\label{eq:c2-g1}
    g_1^{m} \leq g_1 \leq g_1^{M};
    \end{equation*}
       
    \noindent with $g_1^{m}=r_1+d_1(\Delta_1-1)$, $g_1^{M}=r_1+(d_1+1)(\Delta_1-1)$ for some $d_1 \in \mathbb{N}$ such that $c_1+d_1=\left\lceil \frac{\sigma_1}{2}\right \rceil-1$, where $\sigma_1$ is the number of spine vertices of $C_1$.     

    \noindent We prove that $i_2 < g_1^{m}$, which implies $i_2 \neq g_1$. To have $i_2 < g_1^{m}$ it must be:
    \begin{align*}
        &j_2+c_2(\Delta_2-1) < r_1+d_1(\Delta_1-1)\\
        &j_2+c_2(\Delta_2-1) < r_1+qd_1(\Delta_2-1)\\
        &j_2-r_1 < (qd_1-c_2) (\Delta_2-1) \numberthis\label{eq:c2-i1-B}.
    \end{align*}
    \noindent Since $r_l=j_l+\frac{n-(\Delta_l-1)(\sigma_l \mod 2)}{2}$, for $l=1,2$, \Cref{eq:c2-i1-A} can be rewritten as:
    \begin{equation}\label{eq:c2-i1-A'}
        j_2-j_1 = (qc_1-d_2)(\Delta_2-1)-\alpha_2+\frac{n}{2}+\frac{(\Delta_2-1)(\sigma_2 \mod 2)}{2},
    \end{equation}    
    
\noindent while \Cref{eq:c2-i1-B} can be rewritten as:
    \begin{equation}\label{eq:c2-i1-B'}
        j_2-j_1<\left( (qd_1-c_2) -\frac{q(\sigma_1 \mod 2) }{2}\right)(\Delta_2-1)+ \frac{n}{2}.
    \end{equation}
    
\noindent Combining \Cref{eq:c2-i1-A'} and \Cref{eq:c2-i1-B'} we obtain:
    \begin{equation*}
        qc_1-d_2 < (qd_1-c_2) -\frac{(\sigma_2 \mod 2)}{2}-\frac{q(\sigma_1 \mod 2)}{2}+\frac{\alpha_2}{\Delta_2-1}.
    \end{equation*}

\noindent Since $c_l+d_l=\left\lceil \frac{\sigma_l}{2}\right \rceil-1$, we have $d_l= \frac{\sigma_l+1(\sigma_l \mod 2)}{2}-c_l-1$, for $l=1,2$; replacing $d_1$ and $d_2$ in the previous equation, we obtain:
    \begin{align*}
        &qc_1-\frac{\sigma_2}{2}-\frac{\sigma_2 \mod 2}{2}+c_2+1 < \frac{q\sigma_1}{2}+\frac{q(\sigma_1 \mod 2)}{2}-qc_1-q-c_2-\\
        &-\frac{\sigma_2 \mod 2}{2}-\frac{q(\sigma_1 \mod 2)}{2}+\frac{\alpha_2}{\Delta_2-1}
    \end{align*}
\noindent which, considering that $\sigma_2=\frac{n-2}{\Delta_2-1}=\frac{q(n-2)}{\Delta_1-1}=q\sigma_1$, implies:
    \begin{equation}\label{eq:c2-i1-C}
        qc_1-\frac{q\sigma_1}{2}+c_2< -\frac{q+1}{2}+\frac{\alpha_2}{2(\Delta_2-1)}
    \end{equation}

\noindent In summary, to have $i_2^M < g_1^{m}$ \Cref{eq:c2-i1-C} must hold. On the other hand, from \Cref{eq:c2-i1-A'} and from the hypothesis that $j_2-j_1<\frac{n-(\Delta_1-1)}{2}=\frac{n-q(\Delta_2-1)}{2}$ we obtain:
\begin{equation*}\label{eq:c2-i1-D}
        qc_1-d_2 +\frac{1}{2}(\sigma_2 \mod 2)<-\frac{q}{2}+\frac{\alpha_2}{\Delta_2-1}.
\end{equation*}

\noindent Replacing again $d_2$ with $\frac{\sigma_2+1(\sigma_2 \mod 2)}{2}-c_2-1$, we obtain:
\begin{equation}\label{eq:c2-i1-D'}
        qc_1-\frac{q\sigma_1}{2}+c_2<-\frac{q+2}{2}+\frac{\alpha_2}{\Delta_2-1}.
\end{equation}

\noindent We have that $-\frac{q}{2}-1+\frac{\alpha_2}{\Delta_2-1} < -\frac{q}{2}-\frac{1}{2}+\frac{\alpha_2}{2(\Delta_2-1)}$, since $\frac{\alpha_2}{2(\Delta_2-1)} < \frac{1}{2}$. 

In other~words, \Cref{eq:c2-i1-D'} implies that \Cref{eq:c2-i1-C} holds and therefore that $i_2 < g_1^{m}$ and $i_2 \neq g_1$.
     
\end{proof}

\begin{theorem}\label{th:upper-regular-2}
    Let $C_1, C_2, \dots, C_h$ be $h$ caterpillars such that $C_i$ is $\Delta_i$-regular, $\Delta_i-1$ is a multiple of $\Delta_{i+1}-1$, with $1 \le i < h$, and $\Delta_i\leq n-h$ (for $i=1,2,\dots,h$). If $n \geq 2h+(\Delta_1-1)$, then there exists a $k$-planar packing with $k \in O(\Delta_1 h+h^2)$.
\end{theorem}
\begin{proof}
    For each $C_i$, with $1 \leq i \leq h$, we compute a zig-zag drawing $\Gamma_i$ with starting vertex $v_i$. Notice that, given two caterpillars $C_{j_1}$ and $C_{j_2}$ with $1 \leq j_1 < j_2 \le h$, we have that $\Delta_{j_1}-1$ is a multiple of $\Delta_{j_2}-1$, and the zig-zag drawings $\Gamma_{j_1}$ and $\Gamma_{j_2}$ have starting vertices $v_{j_1}$ and $v_{j_2}$, respectively. Hence, $0 < j_2 - j_1 < h$ and by hypothesis $h \leq \frac{n-(\Delta_1-1)}{2}$. Hence, \Cref{le:no-multiple-edges-multiple} holds for every pair of caterpillars and the union of all zig-zag drawings $\Gamma_1, \Gamma_2, \dots, \Gamma_h$ is a valid packing of $C_1, C_2, \dots, C_h$. 

    The proof of the bound on the number of crossings along an edge is the same as the one of \Cref{th:characterization}, considering that $\Delta_l \leq \Delta_1$ and that $j_l-j_1=l-1$ for every $2 \leq l \leq h$. 
\end{proof}

\section{Lower bounds}\label{se:lower-bounds}

In this section we first give a general lower bound on the value of $k$ for $k$-planar $h$-packings; we then increase this lower bound for some small values of $h$.

\begin{theorem}\label{th:lower-bound}
 	Every $k$-planar $h$-packing of $h$ graphs with $n$ vertices and $m$ edges is such that $k \geq \frac{h^2m^2}{14.6n^2}$.  
 \end{theorem}
 \begin{proof}
     The number of edges of a $k$-planar graph with~$n$ vertices is at most $3.81\sqrt{k}\cdot n$, for $k \geq 2$~\cite{ACKERMAN2019101574}. Since the $h$ graphs have $h \cdot m$ edges in total, a $k$-planar packing of these graphs can exist only if $h \leq 3.81\sqrt{k}\frac{n}{m}$, i.e., if $k \geq \frac{h^2m^2}{14.6n^2}$. 
 \end{proof}

Since for a tree $m=n-1$, we have the following.

\begin{corollary}\label{co:lower-bound-trees}
 	Every $k$-planar $h$-packing of $h$ trees is such that $k \geq \frac{h^2}{58.4}$.  
 \end{corollary}

We now refine the lower bound above for small values of $h$ in an $h$-placement of  caterpillars. Specifically we show that for values of $h$ equal to $3$, $4$, and $5$ the corresponding lower bounds are $2$, $3$, and $5$, respectively. Note that for all these cases the lower bound implied by Corollary~\ref{co:lower-bound-trees} is $1$.

\begin{theorem}\label{th:non-delta-regular.1}
    For $h=3,4$ there exists a caterpillar $C$ with at least $h+7$ vertices for which every $k$-planar $h$-placement of $C$ is such that $k \geq h-1$.
    For $h=5$ there exists a caterpillar $C$ with at least $24$ vertices for which every $k$-planar $5$-placement of $C$ is such that $k \geq h$.
\end{theorem}
\begin{proof}
Case $h=3,4$. Let $n$ be an integer such that $n \geq h+7$, and let $C_{n,h}$ be the $n$-vertex caterpillar shown in \Cref{fi:caterpillar}. Notice that the vertex of $C_{n,h}$ denoted as $v$ in \Cref{fi:caterpillar} has degree $n-h$; we call it the \emph{center} of $C_{n,h}$. 
Consider any $h$-placement of $C_{n,h}$ into a graph $G$ and denote as  $v_i$ the vertex of $G$ which the center of $C_i$ is mapped to ($i=1,2,\dots,h$). The vertices $v_1,v_2,\dots,v_h$ must be distinct because, if two centers were mapped to the same vertex of $G$ then this vertex would have degree larger than $n-1$. Namely, if two centers are mapped to the same vertex, this vertex has degree $2n-2h$ which is larger than $n-1$ if $n > 2h-1$, i.e., if $h+7 > 2h-1$, which is true for $h < 6$. Since each $v_i$ ($1 \leq i \leq h$) has degree $n-h$ in $C_i$ and degree $1$ in each of the $h-1$ other caterpillars, its degree in $G$ is $n-1$. Thus, $G$ contains $K_{h,n-h}$. Thus, for $h=3$, $G$ contains $K_{3,7}$ ($n \geq 10$ in this case), which is not $1$-planar~\cite{CZAP2012505}; for $h=4$, $G$ contains $K_{4,7}$ ($n \geq 11$ in this case), which is~not $2$-planar~\cite{JGAA-531}. The case $h=5$ is analogous with $K_{5,19}$, which is not $4$-planar~\cite{ANGELINI201823}.\end{proof}

\begin{figure}[tbp]
		\centering
		\includegraphics[width=0.4\columnwidth, page=1]{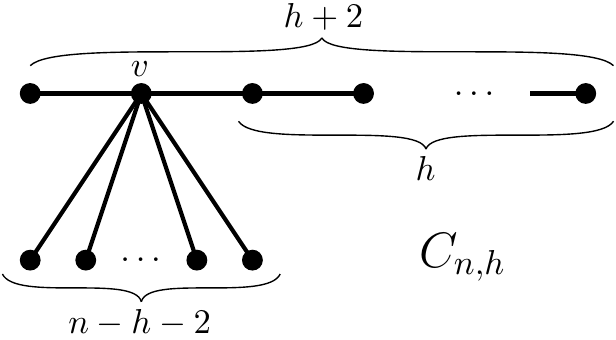}
    \caption{\label{fi:caterpillar} A caterpillar as described in the proof of \Cref{th:non-delta-regular.1}.}
    \end{figure}

\section{Concluding Remarks and Open Problems}\label{se:open-problems}

This paper studied the placement and the  packing of caterpillars into $k$-planar graphs. It proved necessary and sufficient conditions for the $h$-placement of $\Delta$-regular caterpillars in a $k$-planar graph and sufficient conditions for the packing of a set of $\Delta_1$-, $\Delta_2$-, $\dots$, $\Delta_h$-regular caterpillars with $k \in O(\Delta_1 h^2)$ ($\Delta_1$ is the maximum vertex degree in the set). The work in this paper contributes to the rich literature concerning the placement and the packing problem in planar and non-planar host graphs and it specifically relates with a recent re-visitation of these questions in the beyond-planar context. 

\smallskip

Many open problems naturally arise from the research in this paper. We conclude the paper by listing some of those that, in our opinion,  are among the most interesting.

\begin{itemize}
    \item Extend the characterization of Theorem~\ref{th:characterization} to the placement of caterpillars that are not $\Delta$-regular.
    \item Theorems~\ref{th:upper-regular-2} and \ref{th:general-caterpillars} give sufficient conditions for the $k$-planar packing of some families of caterpillars. It would be interesting to give a complete characterization of the packability of these families into $k$-planar graphs.
    \item \Cref{th:non-delta-regular.1} improves the lower bound of \Cref{th:lower-bound} for caterpillars that are not $\Delta$-regular. It would be interesting to find a similar result with $\Delta$-regular caterpillars.
\end{itemize}

Finally, we point out that one could investigate what graphs can be packed/placed into a $k$-planar graph for a given value of $k$,  instead of studying how $k$ varies with the number $h$ and the vertex degree  of the caterpillars that are packed/placed.  While the interested reader can refer to~\cite{ANGELINI201823} for results with $k=1$, the following theorem gives a preliminary result for $k = 2$ (see the appendix for a proof).
Notice that  \Cref{eq:crossings} in the proof of \Cref{th:characterization}  would give upper bounds in the range $[86,137]$ for the caterpillars considered by the following theorem.

\begin{restatable}{theorem}{regularsmall}\label{th:regularsmall}
    A $\Delta$-regular caterpillar with $4\leq \Delta \leq 7$ admits a $2$-planar $3$-placement.
\end{restatable}

\bibliographystyle{plainurl}
\bibliography{biblio}

\clearpage
\appendix

\section{Missing cases for the proof of  Lemma~\ref{le:no-multiple-edges-multiple}}\label{ap:missing-cases}
    
 \medskip \noindent \textbf{Case 1.b: $v_{g_1}$ and $v_{g_2}$ are spine vertices.}
 By \Cref{prop:spine-indices} we have, for some $d_1,d_2 \in \mathbb{N}$:
    \begin{equation}\label{eq:c1b-i1}
        g_1=r_1+d_1(\Delta_1-1)=r_1+qd_1(\Delta_2-1).
    \end{equation}
    and 
    \begin{equation}\label{eq:c1b-i2}
        g_2=r_2+d_2(\Delta_2-1).
    \end{equation}
    
    If $v_{g_1}$ coincides with $v_{g_2}$, we have 
    $g_1=g_2$; from \Cref{eq:c1b-i1} and \Cref{eq:c1b-i2} we obtain:
    \begin{equation}\label{eq:c1b-i1-A}
        r_2-r_1 = (qd_1-d_2)(\Delta_2-1).
    \end{equation}
    
    \noindent Concerning $v_{i_1}$ and $v_{i_2}$, we have:
    \begin{equation*}\label{eq:c1b-g1}
    i_1^{m} \leq i_1 \leq i_1^{M};
    \end{equation*}
    \begin{equation*}\label{eq:c1b-g2}
    i_2^{m} \leq i_2 \leq i_2^{M}.
    \end{equation*}
   
    \noindent with $i_l^{m}=i_l+c_l(\Delta_l-1)$, $i_l^{M}=i_l+(c_l+1)(\Delta_l-1)$ for some $c_l \in \mathbb{N}$ such that $c_l+d_l=\left\lceil \frac{\sigma_l}{2}\right \rceil-1$, where $\sigma_l$ is the number of spine vertices of $C_l$, for $l=1,2$.     

    \noindent We prove that $i_1^{M} < i_2^{m}$, which implies $i_1 \neq i_2$. To have $i_1^{M} < i_2^{m}$ it must be:
    \begin{align*}
        &j_1+(c_1+1)(\Delta_1-1) < j_2+c_2(\Delta_2-1)\\
        &j_1+q(c_1+1)(\Delta_2-1) < j_2+c_2(\Delta_2-1)\\
        &j_2-j_1 > (qc_1+q-c_2) (\Delta_2-1)\numberthis \label{eq:c1b-i1-B}.
    \end{align*}
    \noindent Since $r_l=j_l+\frac{n-(\Delta_l-1)(\sigma_l \mod 2)}{2}$, for $l=1,2$, \Cref{eq:c1b-i1-A} can be rewritten as:
    \begin{equation}\label{eq:c1b-i1-A'}
        j_2-j_1=\left ((qd_1-d_2) +\frac{(\sigma_2 \mod 2)}{2}-\frac{q(\sigma_1 \mod 2)}{2}\right )(\Delta_2-1).
    \end{equation}
\noindent Combining \Cref{eq:c1b-i1-A'} and \Cref{eq:c1b-i1-B} we obtain:
    \begin{equation}\label{eq:c1b-i1-C}
        c_2-qc_1-q> -\frac{1}{2}.
    \end{equation}

\noindent In summary, to have $i_1^{M} < i_2^{m}$ \Cref{eq:c1b-i1-C} must hold. On the other hand, from \Cref{eq:c1b-i1-A'} and from the hypothesis that $j_2-j_1>0$ we obtain $c_2-qc_1-q>-1$ which, since $c_2-qc_1-q$ is integer, can be rewritten as $c_2-qc_1-q\geq 0$. This implies that \Cref{eq:c1b-i1-C} holds and therefore that $i_1^{M} < i_2^{m}$ and $i_1 \neq i_2$.
    
 \medskip \noindent \textbf{Case 1.c: $v_{i_1}$ and $v_{g_2}$ are spine vertices.}
By \Cref{prop:spine-indices} we have, for some $c_1 \in \mathbb{N}$:
    \begin{equation}\label{eq:c1c-i1}
        i_1=j_1+c_1(\Delta_1-1)=j_1+qc_1(\Delta_2-1).
    \end{equation}
We also have, for some $c_2 \in \mathbb{N}$ and $0 \leq \alpha_2 < \Delta_2-1$:
    \begin{equation}\label{eq:c1c-i2}
        i_2=j_2+c_2(\Delta_2-1) + \alpha_2.
    \end{equation}
    
    If $v_{i_1}$ coincides with $v_{i_2}$, we have 
    $i_1=i_2$; from \Cref{eq:c1c-i1} and \Cref{eq:c1c-i2} we obtain:
    \begin{equation}\label{eq:c1c-i1-A}
        j_2-j_1 = (qc_1-c_2)(\Delta_2-1) - \alpha_2.
    \end{equation}
    
    \noindent Concerning $v_{g_1}$, we have:
    \begin{equation*}\label{eq:c1c-g1}
    g_1^{m} \leq g_1 \leq g_1^{M};
    \end{equation*}

    \noindent with $g_1^{m}=r_1+d_1(\Delta_1-1)$, $g_1^{M}=r_1+(d_1+1)(\Delta_1-1)$ for some $d_1 \in \mathbb{N}$ such that $c_1+d_1=\left\lceil \frac{\sigma_1}{2}\right \rceil-1$, where $\sigma_1$ is the number of spine vertices of $C_1$.
    
    \noindent Since $v_{g_2}$ is a vertex in the lower part of $\Gamma_2$, it must be $g_2 = r_2 + d_2(\Delta_2 - 1)$.

    \noindent We prove that $g_1^{M} < g_2$, which implies $g_1 \neq g_2$. To have $g_1^{M} < g_2$ it must be:
    \begin{align*}
        &r_1+(d_1+1)(\Delta_1-1) < r_2+d_2(\Delta_2-1)\\
        &r_1+q(d_1+1)(\Delta_2-1) < r_2+d_2(\Delta_2-1)\\
        &r_2-r_1 > (qd_1+q-d_2) (\Delta_2-1)\numberthis\label{eq:c1c-i1-B}.
    \end{align*}
    \noindent Since $r_l=j_l+\frac{n-(\Delta_l-1)(\sigma_l \mod 2)}{2}$, for $l=1,2$, \Cref{eq:c1c-i1-B} can be rewritten as:
    \begin{equation}\label{eq:c1c-i1-B'}
        j_2-j_1>\left( (qd_1+q-d_2) +\frac{(\sigma_2 \mod 2)}{2}-\frac{q(\sigma_1 \mod 2) }{2}\right)(\Delta_2-1).
    \end{equation}
\noindent Combining \Cref{eq:c1c-i1-A} and \Cref{eq:c1c-i1-B'} we obtain:
    \begin{equation*}
        qc_1-c_2-\frac{\alpha_2}{\Delta_2-1}> (qd_1+q-d_2) +\frac{(\sigma_2 \mod 2)}{2}-\frac{q(\sigma_1 \mod 2) }{2}.
    \end{equation*}

\noindent Since $c_l+d_l=\left\lceil \frac{\sigma_l}{2}\right \rceil-1$, we have $d_l= \frac{\sigma_l+1(\sigma_l \mod 2)}{2}-c_l-1$, for $l=1,2$; replacing $d_1$ and $d_2$ in the previous equation, we obtain:

    \begin{align*}
        &qc_1-c_2-\frac{\alpha_2}{\Delta_2-1}> \frac{q\sigma_1}{2}+\frac{q(\sigma_1 \mod 2)}{2}-qc_1-q+q-\frac{\sigma_2}{2}-\frac{1(\sigma_2 \mod 2)}{2}+c_2+\\
        &+1+\frac{\sigma_2 \mod 2}{2}-\frac{q(\sigma_1 \mod 2)}{2}
    \end{align*}
\noindent which, considering that $\sigma_2=\frac{n-2}{\Delta_2-1}=\frac{q(n-2)}{\Delta_1-1}=q\sigma_1$, implies:
    \begin{equation}\label{eq:c1c-i1-C}
        qc_1-c_2> \frac{1}{2} + \frac{\alpha_2}{2(\Delta_2-1)}
    \end{equation}

\noindent In summary, to have $g_1^{M} < g_2$ \Cref{eq:c1c-i1-C} must hold. On the other hand, from \Cref{eq:c1c-i1-A} and from the hypothesis that $j_2-j_1>0$ we obtain $(qc_1-c_2)(\Delta_2-1)-\alpha_2>0$ which, since $(\Delta_2-1)>0$, implies $qc_1-c_2>\frac{\alpha_2}{\Delta_2-1}$.
Since $0 \leq \frac{\alpha_2}{\Delta_2-1} <1$ and $qc_1-c_2$ is integer, we have $qc_1-c_2\geq 1$.
This implies that \Cref{eq:c1c-i1-C} holds and therefore that $g_1^{M} < g_2$ and $g_1 \neq g_2$.
    
 \medskip \noindent \textbf{Case 1.d: $v_{g_1}$ and $v_{i_2}$ are spine vertices.}
 By \Cref{prop:spine-indices} we have, for some $c_2 \in \mathbb{N}$:
    \begin{equation}\label{eq:c1d-i2}
        i_2=j_2+c_2(\Delta_2-1).
    \end{equation}

    We also have, for some $c_1 \in \mathbb{N}$ and $0 \leq \alpha_1 < \Delta_2-1$:
    \begin{equation}\label{eq:c1d-i1}
        i_1=j_1+c_1(\Delta_1-1)+\alpha_1=j_1+qc_1(\Delta_2-1)+\alpha_1.
    \end{equation}
    
    If $v_{i_1}$ coincides with $v_{i_2}$, we have 
    $i_1=i_2$; from \Cref{eq:c1d-i1} and \Cref{eq:c1d-i2} we obtain:
    \begin{equation}\label{eq:c1d-i1-A}
        j_2-j_1 = (qc_1-c_2)(\Delta_2-1) + \alpha_1.
    \end{equation}
    
    \noindent Concerning $v_{g_2}$, we have:
    \begin{equation*}\label{eq:c1d-g2}
    g_2^{m} \leq g_2 \leq g_2^{M}.
    \end{equation*}
   
    \noindent with $g_2^{m}=r_2+d_2(\Delta_2-1)$, $g_2^{M}=r_2+(d_2+1)(\Delta_2-1)$ for some $d_2 \in \mathbb{N}$ such that $c_2+d_2=\left\lceil \frac{\sigma_2}{2}\right \rceil-1$, where $\sigma_2$ is the number of spine vertices of $C_2$.
    
    \noindent Since $v_{g_1}$ is a vertex in the lower part of $\Gamma_1$, it must be $g_1 = r_1 + d_1(\Delta_1 - 1)$.

    \noindent We prove that $g_1 < g_2^{m}$, which implies $g_1 \neq g_2$. To have $g_1 < g_2^{m}$ it must be:
    \begin{align*}
        &r_1+d_1(\Delta_1-1) < r_2+d_2(\Delta_2-1)\\
        &r_1+qd_1(\Delta_2-1) < r_2+d_2(\Delta_2-1)\\
        &r_2-r_1 > (qd_1-d_2) (\Delta_2-1)\numberthis\label{eq:c1d-i1-B}.
    \end{align*}
    \noindent Since $r_l=j_l+\frac{n-(\Delta_l-1)(\sigma_l \mod 2)}{2}$, for $l=1,2$, \Cref{eq:c1d-i1-B} can be rewritten as:
    \begin{equation}\label{eq:c1d-i1-B'}
        j_2-j_1>\left( (qd_1-d_2) +\frac{(\sigma_2 \mod 2)}{2}-\frac{q(\sigma_1 \mod 2) }{2}\right)(\Delta_2-1).
    \end{equation}
\noindent Combining \Cref{eq:c1d-i1-A} and \Cref{eq:c1d-i1-B'} we obtain:
    \begin{equation*}
        qc_1-c_2+\frac{\alpha_1}{\Delta_2-1}> (qd_1-d_2) +\frac{(\sigma_2 \mod 2)}{2}-\frac{q(\sigma_1 \mod 2) }{2}.
    \end{equation*}

\noindent Since $c_l+d_l=\left\lceil \frac{\sigma_l}{2}\right \rceil-1$, we have $d_l= \frac{\sigma_l+1(\sigma_l \mod 2)}{2}-c_l-1$, for $l=1,2$; replacing $d_1$ and $d_2$ in the previous equation, we obtain:
    \begin{align*}
        &qc_1-c_2+\frac{\alpha_1}{\Delta_2-1}> \frac{q\sigma_1}{2}+\frac{q(\sigma_1 \mod 2)}{2}-qc_1-q-\frac{\sigma_2}{2}-\frac{1(\sigma_2 \mod 2)}{2}+c_2+\\
        &+1+\frac{\sigma_2 \mod 2}{2}-\frac{q(\sigma_1 \mod 2)}{2}
    \end{align*}
\noindent which, considering that $\sigma_2=\frac{n-2}{\Delta_2-1}=\frac{q(n-2)}{\Delta_1-1}=q\sigma_1$, implies:
    \begin{equation}\label{eq:c1d-i1-C}
        qc_1-c_2> \frac{1-q}{2} - \frac{\alpha_1}{2(\Delta_2-1)}
    \end{equation}

\noindent In summary, to have $g_1 < g_2^{m}$ \Cref{eq:c1d-i1-C} must hold. On the other hand, from \Cref{eq:c1d-i1-A} and from the hypothesis that $j_2-j_1>0$ we obtain $(qc_1-c_2)(\Delta_2-1)+\alpha_1>0$ which, since $(\Delta_2-1)>0$, implies $qc_1-c_2>-\frac{\alpha_1}{\Delta_2-1}$.
Since $0 \leq \frac{\alpha_1}{\Delta_2-1} <1$ and $qc_1-c_2$ is integer, we have $qc_1-c_2> 0$. Since $q$ is a positive integer, this implies that \Cref{eq:c1d-i1-C} holds and therefore that $g_1 < g_2^{m}$ and $g_1 \neq g_2$.

\medskip \noindent \textbf{Case 2.b: $v_{g_1}$ and $v_{g_2}$ are spine vertices.}
Since $v_{g_2}$ is a vertex in the lower part of $\Gamma_2$, it must be $g_2=r_2+d_2(\Delta_2-1)$. If $v_{g_2}$ coincides with $v_{i_1}$, as explained above, it must be $i_1=g_2-n$. Combining the expression of $g_2$ with \Cref{eq:c1d-i1} we obtain:
    \begin{equation}\label{eq:c2b-i1-A}
        r_2-j_1 = (qc_1-d_2)(\Delta_2-1)+\alpha_1+n.
    \end{equation}
    
    \noindent Concerning $v_{i_2}$, we have:
    \begin{equation*}\label{eq:c2b-g1}
    i_2^{m} \leq i_2 \leq i_2^{M};
    \end{equation*}
       
    \noindent with $i_2^{M}=j_2+(c_2+1)(\Delta_2-1)$ for some $c_2 \in \mathbb{N}$ such that $c_2+d_2=\left\lceil \frac{\sigma_2}{2}\right \rceil-1$, where $\sigma_2$ is the number of spine vertices of $C_2$.     

    \noindent We prove that $i_2^{M} < g_1$, which implies $i_2 \neq g_1$. To have $i_2^{M} < g_1$ it must be:
    \begin{align*}
        &j_2+(c_2+1)(\Delta_2-1) < r_1+d_1(\Delta_1-1)\\
        &j_2+(c_2+1)(\Delta_2-1) < r_1+qd_1(\Delta_2-1)\\
        &j_2-r_1 < (qd_1-c_2-1) (\Delta_2-1)\numberthis\label{eq:c2b-i1-B}.
    \end{align*}
    \noindent Since $r_l=j_l+\frac{n-(\Delta_l-1)(\sigma_l \mod 2)}{2}$, for $l=1,2$, \Cref{eq:c2b-i1-A} can be rewritten as:
    \begin{equation}\label{eq:c2b-i1-A'}
        j_2-j_1 = (qc_1-d_2)(\Delta_2-1)/\alpha_1+\frac{n}{2}+\frac{(\Delta_2-1)(\sigma_2 \mod 2)}{2},
    \end{equation}    
    
\noindent while \Cref{eq:c2b-i1-B} can be rewritten as:
    \begin{equation}\label{eq:c2b-i1-B'}
        j_2-j_1<\left( (qd_1-c_2-1) -\frac{q(\sigma_1 \mod 2) }{2}\right)(\Delta_2-1)+ \frac{n}{2}.
    \end{equation}
    
\noindent Combining \Cref{eq:c2b-i1-A'} and \Cref{eq:c2b-i1-B'} we obtain:
    \begin{equation*}
        qc_1-d_2 < (qd_1-c_2-1) -\frac{(\sigma_2 \mod 2)}{2}-\frac{q(\sigma_1 \mod 2)}{2}-\frac{\alpha_1}{\Delta_2-1}.
    \end{equation*}

\noindent Since $c_l+d_l=\left\lceil \frac{\sigma_l}{2}\right \rceil-1$, we have $d_l= \frac{\sigma_l+1(\sigma_l \mod 2)}{2}-c_l-1$, for $l=1,2$; replacing $d_1$ and $d_2$ in the previous equation, we obtain:

    \begin{align*}
        &qc_1-\frac{\sigma_2}{2}-\frac{\sigma_2 \mod 2}{2}+c_2+1+\frac{\sigma_2 \mod 2}{2}+\frac{\alpha_1}{\Delta_2-1} < \frac{q\sigma_1}{2}+\frac{q(\sigma_1 \mod 2)}{2}-\\&-qc_1-q-c_2-\frac{q(\sigma_1 \mod 2)}{2}-1
    \end{align*}
\noindent which, considering that $\sigma_2=\frac{n-2}{\Delta_2-1}=\frac{q(n-2)}{\Delta_1-1}=q\sigma_1$, implies:
    \begin{equation}\label{eq:c2b-i1-C}
        qc_1-\frac{q\sigma_1}{2}+c_2+1< -\frac{q}{2}-\frac{\alpha_1}{2(\Delta_2-1)}
    \end{equation}

\noindent In summary, to have $i_2^M < g_1^{m}$ \Cref{eq:c2b-i1-C} must hold. On the other hand, from \Cref{eq:c2b-i1-A'} and from the hypothesis that $j_2-j_1<\frac{n-(\Delta_1-1)}{2}=\frac{n-q(\Delta_2-1)}{2}$ we obtain:
\begin{equation}\label{eq:c2b-i1-D}
        qc_1-\frac{q\sigma_1}{2} +c_2+1<-\frac{q}{2}-\frac{\alpha_1}{\Delta_2-1}.
\end{equation}

We have that $-\frac{q}{2}-\frac{\alpha_1}{\Delta_2-1} < -\frac{q}{2}-\frac{\alpha_1}{2(\Delta_2-1)}$, since $\frac{\alpha_1}{2(\Delta_2-1)}<\frac{1}{2}$.
In other words, \Cref{eq:c2b-i1-D} implies that \Cref{eq:c2b-i1-C} holds and therefore that $i_2^M < g_1$ and $i_2 \neq g_1$.

 \medskip \noindent \textbf{Case 2.c: $v_{i_1}$ and $v_{g_2}$ are spine vertices.}
Since $v_{g_2}$ is a vertex in the lower part of $\Gamma_2$, it must be $g_2=r_2+d_2(\Delta_2-1)$. If $v_{g_2}$ coincides with $v_{i_1}$, as explained above, it must be $i_1=g_2-n$. Combining the expression of $g_2$ with \Cref{eq:c1c-i1} we obtain:
    \begin{equation}\label{eq:c2c-i1-A}
        r_2-j_1 = (qc_1-d_2)(\Delta_2-1)+n.
    \end{equation}
    
    \noindent Concerning $v_{g_1}$, we have:
    \begin{equation*}\label{eq:c2c-g1}
    g_1^{m} \leq g_1 \leq g_1^{M};
    \end{equation*}
       
    \noindent with $g_1^{m}=r_1+d_1(\Delta_1-1)$, $g_1^{M}=r_1+(d_1+1)(\Delta_1-1)$ for some $d_1 \in \mathbb{N}$ such that $c_1+d_1=\left\lceil \frac{\sigma_1}{2}\right \rceil-1$, where $\sigma_1$ is the number of spine vertices of $C_1$.     

    \noindent We prove that $i_2^{M} < g_1^{m}$, which implies $i_2 \neq g_1$. To have $i_2^{M} < g_1^{m}$ it must be:
    \begin{align*}
        &j_2+(c_2+1)(\Delta_2-1) < r_1+d_1(\Delta_1-1)\\
        &j_2+(c_2+1)(\Delta_2-1) < r_1+qd_1(\Delta_2-1)\\
        &j_2-r_1 < (qd_1-c_2-1) (\Delta_2-1)\numberthis\label{eq:c2c-i1-B}.
    \end{align*}
    \noindent Since $r_l=j_l+\frac{n-(\Delta_l-1)(\sigma_l \mod 2)}{2}$, for $l=1,2$, \Cref{eq:c2c-i1-A} can be rewritten as:
    \begin{equation}\label{eq:c2c-i1-A'}
        j_2-j_1 = (qc_1-d_2)(\Delta_2-1)+\frac{n}{2}+\frac{(\Delta_2-1)(\sigma_2 \mod 2)}{2},
    \end{equation}    
    
\noindent while \Cref{eq:c2c-i1-B} can be rewritten as:
    \begin{equation}\label{eq:c2c-i1-B'}
        j_2-j_1<\left( (qd_1-c_2-1) -\frac{q(\sigma_1 \mod 2) }{2}\right)(\Delta_2-1)+ \frac{n}{2}.
    \end{equation}
    
\noindent Combining \Cref{eq:c2c-i1-A'} and \Cref{eq:c2c-i1-B'} we obtain:
    \begin{equation*}
        qc_1-d_2 < (qd_1-c_2-1) -\frac{(\sigma_2 \mod 2)}{2}-\frac{q(\sigma_1 \mod 2)}{2}.
    \end{equation*}

\noindent Since $c_l+d_l=\left\lceil \frac{\sigma_l}{2}\right \rceil-1$, we have $d_l= \frac{\sigma_l+1(\sigma_l \mod 2)}{2}-c_l-1$, for $l=1,2$; replacing $d_1$ and $d_2$ in the previous equation, we obtain:
    \begin{align*}
        &qc_1-\frac{\sigma_2}{2}-\frac{\sigma_2 \mod 2}{2}+c_2+1 < \frac{q\sigma_1}{2}+\frac{q(\sigma_1 \mod 2)}{2}-qc_1-q-c_2-1-\\
        &-\frac{\sigma_2 \mod 2}{2}-\frac{q(\sigma_1 \mod 2)}{2}
    \end{align*}
\noindent which, considering that $\sigma_2=\frac{n-2}{\Delta_2-1}=\frac{q(n-2)}{\Delta_1-1}=q\sigma_1$, implies:
    \begin{equation}\label{eq:c2c-i1-C}
        qc_1-\frac{q\sigma_1}{2}+c_2+1< -\frac{q}{2}
    \end{equation}

\noindent In summary, to have $i_2^M < g_1^{m}$ \Cref{eq:c2c-i1-C} must hold. On the other hand, from \Cref{eq:c2c-i1-A'} and from the hypothesis that $j_2-j_1<\frac{n-(\Delta_1-1)}{2}=\frac{n-q(\Delta_2-1)}{2}$ we obtain:
\begin{equation*}\label{eq:c2c-i1-D}
        qc_1-d_2 +\frac{1}{2}(\sigma_2 \mod 2)<-\frac{q}{2}.
\end{equation*}

\noindent Replacing again $d_2$ with $\frac{\sigma_2+1(\sigma_2 \mod 2)}{2}-c_2-1$, we obtain:
\begin{equation}\label{eq:c2c-i1-D'}
        qc_1-\frac{q\sigma_1}{2}+c_2+1<-\frac{q}{2}.
\end{equation}

Since \Cref{eq:c2c-i1-C} is equivalent to \Cref{eq:c2c-i1-D'}, we can conclude that \Cref{eq:c2c-i1-C} holds and therefore that $i_2^M < g_1^{m}$ and $i_2 \neq g_1$.
    
 \medskip \noindent \textbf{Case 2.d: $v_{g_1}$ and $v_{i_2}$ are spine vertices.}
 Since $v_{g_1}$ is a vertex in the lower part of $\Gamma_1$, it must be $g_1=r_1+d_1(\Delta_1-1)$. If $v_{g_1}$ coincides with $v_{i_2}$, combining the expression of $g_1$ with \Cref{eq:c1-i2} we obtain:
    \begin{equation}\label{eq:c2d-i1-A}
        j_2-r_1 = (qd_1-c_2)(\Delta_2-1).
    \end{equation}
    
    \noindent Concerning $v_{i_1}$ $v_{g_2}$, we have:
    \begin{equation*}
    i_1^{m} \leq i_1 \leq i_1^{M};
    \end{equation*}
    and
    \begin{equation*}
    g_2^{m} \leq g_2 \leq g_2^{M}.
    \end{equation*}
       
    \noindent with $i_1^{m}=j_1+c_1(\Delta_1-1)=j_1+qc_1(\Delta_2-1)$, $g_2^{M}=r_2+(d_2+1)(\Delta_2-1)-n$ for some $d_2 \in \mathbb{N}$ such that $c_2+d_2=\left\lceil \frac{\sigma_2}{2}\right \rceil-1$, where $\sigma_2$ is the number of spine vertices of $C_2$.     

    \noindent We prove that $g_2^{M} < i_1^{m}$, which implies $g_2 \neq i_1$. To have $g_2^{M} < i_1^{m}$ it must be:
    \begin{align*}
        &r_2+(d_2+1)(\Delta_2-1)-n < j_1+c_1(\Delta_1-1)\\
        &r_2+(d_2+1)(\Delta_2-1)-n < j_1+qc_1(\Delta_2-1)\\
        &r_2-j_1 < (qc_1-d_2-1)(\Delta_2-1)+n (\Delta_2-1)\numberthis\label{eq:c2d-i1-B}.
    \end{align*}
    \noindent Since $r_l=j_l+\frac{n-(\Delta_l-1)(\sigma_l \mod 2)}{2}$, for $l=1,2$, \Cref{eq:c2d-i1-A} can be rewritten as:
    \begin{equation}\label{eq:c2d-i1-A'}
        j_2-j_1 = (qd_1-c_2)(\Delta_2-1)-\frac{q(\Delta_2-1)(\sigma_1 \mod 2)}{2}+\frac{n}{2},
    \end{equation}    
    
\noindent while \Cref{eq:c2d-i1-B} can be rewritten as:
    \begin{equation}\label{eq:c2d-i1-B'}
        j_2-j_1< (qc_1-d_2-1)(\Delta_2-1) +\frac{(\Delta_2-1)(\sigma_2 \mod 2) }{2}+ \frac{n}{2}.
    \end{equation}
    
\noindent Combining \Cref{eq:c2d-i1-A'} and \Cref{eq:c2d-i1-B'} we obtain:
    \begin{equation*}
        qd_1-c_2 -\frac{q(\sigma_1 \mod 2)}{2} < (qc_1-d_2-1) +\frac{(\sigma_2 \mod 2)}{2}.
    \end{equation*}

\noindent Since $c_l+d_l=\left\lceil \frac{\sigma_l}{2}\right \rceil-1$, we have $d_l= \frac{\sigma_l+1(\sigma_l \mod 2)}{2}-c_l-1$, for $l=1,2$; replacing $d_1$ and $d_2$ in the previous equation, we obtain:
    \begin{align*}
        &\frac{q\sigma_1}{2}+\frac{q(\sigma_1 \mod 2)}{2}-qc_1-q-c_2-\frac{q(\sigma_1 \mod 2)}{2} < qc_1- \frac{\sigma_2}{2}-\frac{\sigma_2 \mod 2}{2}+\\
        &+c_2+1-1 +\frac{\sigma_2 \mod 2}{2}
    \end{align*}
\noindent which, considering that $\sigma_2=\frac{n-2}{\Delta_2-1}=\frac{q(n-2)}{\Delta_1-1}=q\sigma_1$, implies:
    \begin{equation}\label{eq:c2d-i1-C}
        2 \left (qc_1-\frac{q\sigma_1}{2}+c_2 \right)> -q
    \end{equation}

\noindent In summary, to have $g_2^M < i_1^{m}$ \Cref{eq:c2d-i1-C} must hold. On the other hand, from \Cref{eq:c2d-i1-A'} and from the hypothesis that $j_2-j_1<\frac{n-(\Delta_1-1)}{2}=\frac{n-q(\Delta_2-1)}{2}$ we obtain:
\begin{equation*}\label{eq:c2d-i1-D}
        qd_1-c_2 -\frac{q}{2}(\sigma_1 \mod 2)<-\frac{q}{2}.
\end{equation*}

\noindent Replacing again $d_1$ with $\frac{\sigma_1+1(\sigma_1 \mod 2)}{2}-c_1-1$, we obtain:
\begin{equation}\label{eq:c2d-i1-D'}
        2 \left (qc_1-\frac{q\sigma_1}{2}+c_2 \right)>-q.
\end{equation}

Since \Cref{eq:c2d-i1-C} is equivalent to \Cref{eq:c2d-i1-D'}, we can conclude that \Cref{eq:c2d-i1-C} holds and therefore that $g_2^M < i_1^{m}$ and $g_2 \neq i_1$.

\section{Proof of Theorem~\ref{th:regularsmall}}

\regularsmall*
\begin{proof}
    Let $C_1$, $C_2$, and $C_3$ be three copies (shown in red, blue and green, respectively, in \Cref{fi:2-planar-k-ary-caterpillar}) of a $\Delta$-regular caterpillar $C$ with $4 \leq \Delta \leq 7$. We denote the vertices of caterpillar $C_j$ for $j=1,2,3$ as follows; the spine vertices are denoted as $v^j_0, v^j_1, \dots, v^j_{c-1}$ in the order they appear along the spine; the leaves adjacent to vertex $v^j_i$ (for $i=1,2,\dots,c-1$) are denoted as $u^j_{i,l}$ with $l=0,1,\dots,d$, where $d= \Delta-2$ if $i=0$ or $i=c-1$ and $d=\Delta-3$ if $0 < i < c-1$. 
	
	Let $p_0, p_1, \dots,p_{n-1}$ be $n$ points on a circle in clockwise order (with indices taken modulo $n$). To construct the packing, we compute a drawing for each caterpillar such that the vertices are mapped to points $p_1, p_2, \dots,p_n$ and the union of the three drawings is a $2$-planar drawing. We describe the construction for $\Delta=4,5,6$ (see also \Cref{fi:2-planar-k-ary-caterpillar-b,fi:2-planar-k-ary-caterpillar-c,fi:2-planar-k-ary-caterpillar-d}); the construction in the case $\Delta=7$ is slightly different and it is shown in \Cref{fi:2-planar-k-ary-caterpillar-e}.
	
	Caterpillar $C_1$ is drawn outside the circle so that vertex $v^1_0$ is mapped to point $p_0$, each vertex $v^1_i$, for $i=1,2,\dots, c-1$ is mapped to $p_{i(\Delta-1)+1}$, each leaf $u^1_{0,l}$ is mapped to the point $p_{l+1}$, and each leaf $u^1_{i,l}$ is mapped to the point $p_{i(\Delta-1)+2+l}$.
	In other words, each vertex of the spine is followed clockwise by its leaves and the last of these leaves is followed by the next vertex of the spine. Caterpillar $C_2$ is drawn inside the circle so that vertex $v^2_i$ is mapped to the point immediately following clockwise the point hosting $v^1_i$ and each leaf $u^2_{i,l}$ is mapped to the point immediately following clockwise $u^1_{i,l}$.  Clearly, the drawings of the first two caterpillars do not cross each other because they are on different sides of the circle; also, their union has no multiple edges. Concerning $C_3$, the vertex $v^3_i$, for $i=0,1,\dots,c-2$ is mapped to the point that hosts $u^1_{i,d}$ and $u^2_{i,d-1}$, i.e., the last leaf of $v^1_i$ and the second last leaf of $v^2_i$; the vertex $v^3_{c-1}$  is mapped to the point that hosts $u^1_{i,d-1}$ and $u^2_{i,d-2}$, i.e., the second last leaf of $v^1_{c-1}$ and the third last leaf of $v^2_i$. About this mapping, observe that if we draw the edges of the spine of $C_3$ outside the circle, each edge of the spine of $C_3$ intersects two consecutive edges of the spine of $C_1$ and each edge of the spine of $C_1$ intersects at most two consecutive edges of the spine of $C_3$. To complete the drawing, we need to draw the leaves of $C_3$. Consider two consecutive spine vertices $v^3_i$ and $v^3_{i+1}$, with $0 \leq i \leq c-2$; between these two vertices there are $\Delta-2$ points not yet used by $C_3$, we connect the first two of these vertices in clockwise order to $v_i$. Depending on the value of $\Delta$, there remain $0$, $1$, or $2$ points between $v_i$ and $v_{i+1}$ not yet used by $C_3$;  we connect these points to $v_{i+1}$. Notice that, there remain to map $\Delta-3$ leaves adjacent to $v^3_0$ and $3$ leaves adjacent to $v^3_{c-1}$. On the other hand, there are $\Delta$ points not yet used by $C_3$ that are between $v^3_{c-1}$ and $v^3_0$ clockwise; we connect the three vertices following clockwise $v^3_{c-1}$ to $v^3_{c-1}$, and the remaining ones to $v^3_0$. All the edges of $C_3$ that are incident to leaves are drawn inside the circle. 
	This mapping of $C_3$ does not create multiple edges and gives rise to at most two crossings along the edges of $C_2$ and $C_3$. 
\end{proof}

\begin{figure}[tbp]
	\centering
	\subfigure[]{\label{fi:2-planar-k-ary-caterpillar-b}\includegraphics[width=0.24\columnwidth, page=7]{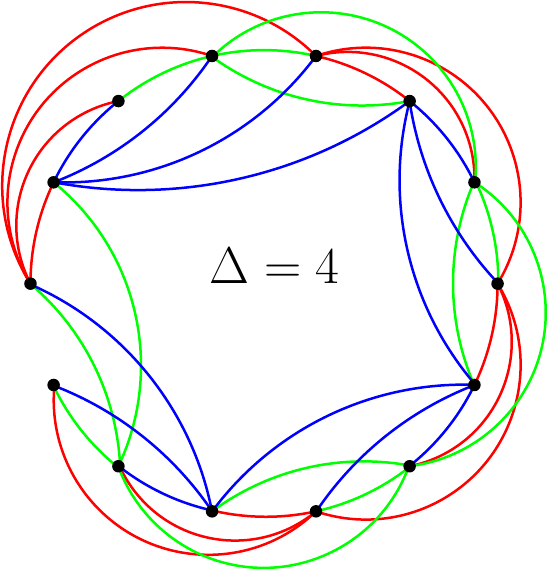}}\hfil 
	\subfigure[]{\label{fi:2-planar-k-ary-caterpillar-c}\includegraphics[width=0.24\columnwidth, page=8]{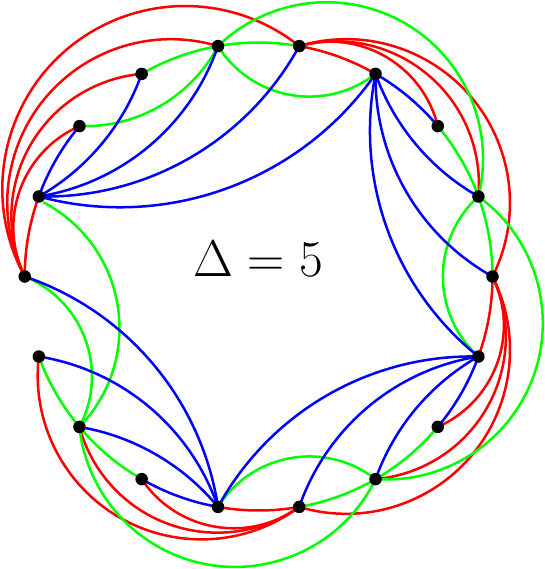}}
	\subfigure[]{\label{fi:2-planar-k-ary-caterpillar-d}\includegraphics[width=0.24\columnwidth, page=9]{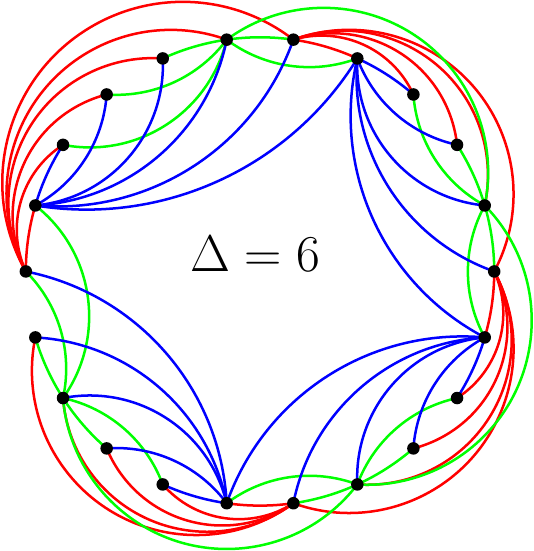}}\hfil
	\subfigure[]{\label{fi:2-planar-k-ary-caterpillar-e}\includegraphics[width=0.24\columnwidth, page=10]{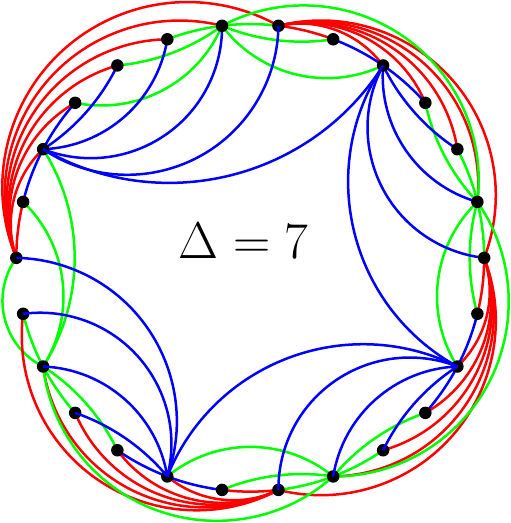}}
	\caption{\label{fi:2-planar-k-ary-caterpillar} $2$-planar $3$-placements of $\Delta$-regular caterpillars.}
\end{figure}
\end{document}